\documentclass[12pt,reqno]{amsart}
\usepackage{tikz-cd}
\usepackage{amscd}
\usepackage{amsmath, amsfonts,amssymb,amsthm,mathrsfs,xypic}
\usepackage{graphicx}
\usepackage{fancybox}
\calclayout
\usepackage{amsmath}
\usepackage{hyperref}
\usepackage{cleveref}
\usepackage[all]{xy}
\usepackage{extarrows}

\usepackage{mathtools}
\usepackage{mathrsfs,stmaryrd,wasysym}
\usepackage{graphics}
\usepackage{xypic}
\usepackage{multirow}
\usepackage{makecell}

\usepackage[english]{babel}
\usepackage{amsmath}

\allowdisplaybreaks[4]

\usepackage{cases}

\usepackage{dsfont}
\usepackage{color}
\SelectTips{cm}{}
\calclayout

\numberwithin{equation}{section}

\theoremstyle{plain}
\newtheorem{theorem}[equation]{Theorem}
\newtheorem{proposition}[equation]{Proposition}
\newtheorem{lemma}[equation]{Lemma}
\newtheorem{corollary}[equation]{Corollary}

\theoremstyle{definition}
\newtheorem{definition}[equation]{Definition}

\theoremstyle{remark}
\newtheorem{remark}[equation]{Remark}

\def\mbZ{\mathbb Z}
\def\mcZ{\mathcal{Z}}
\def\mcF{\mathcal{F}}

\def\nslash{\:\notslash\:}
\def\mcH{\mathcal H}
\def\mcC{\mathcal C}
\def\mcR{\mathcal R}

\def\mfD{\mathfrak{D}}
\def\mcW{\mathcal W}
\def\parti{\partial}

\def\bi{\text{\boldmath$i$}}

\def\bj{\text{\boldmath$j$}}

\begin{document}
	
	\title[The Brundan-Kleshchev subalgebra and BK-type isomorphism]{The Brundan-Kleshchev subalgebra and BK-type isomorphism}
	\author [\tiny{Fan Kong and Zhi-Wei Li}] {Fan Kong and Zhi-Wei Li*}
	
	\begin{abstract}  We use a unified method to give an isomorphism between direct sums of cyclotomic affine (and degenerate affine) Hecke algebras and cyclotomic BK-subalgebras which are some KLR-type algebras.
	\end{abstract}
	
	\date{\today}
	\thanks{*Corresponding author}
	\thanks{{\em 2000 Mathematics Subject Classification:} 20C08.}
	\thanks{The first author is supported by the NSF of China (No. 11501057) and the Doctoral Fund of Southwest University (SWU-118003); The second author is supported by the NSF	of China (No. 11671174).}
	
	\address{Fan Kong\\ School of Mathematics and Statistics \\
		Southwest University \\ Chongqing 400715\\ P. R. China.}
	\email{kongfan85@126.com}

	\address{Zhi-Wei Li\\ School of Mathematics and Statistics \\
		Jiangsu Normal University\\	Xuzhou 221116 Jiangsu \\ P. R. China.}
	\email{zhiweili@jsnu.edu.cn}
	\maketitle
	
	\setcounter{tocdepth}{1}
	
	\section{Introduction}
	
 Brundan and Kleshchev \cite{Brundan-Kleshchev09} gave an explicit isomorphism between blocks of cyclotimic affine Hecke algebras of symmetric groups and cyclotomic KLR algebras of type A. But their proof is complicated since the invalidity of some rational functions.

 On the other hand, Lusztig has showed that there is a natural isomorphism between affine Hecke algebras of all types and their graded versions \cite{Lusztig89}. One key point is the rationalization of affine Hecke algebras. Motivated by Lusztig's work, we introduce the Lusztig extension of affine Hecke algebras of all types and the Brundan-Kleshchev subalgebra (or BK subalgebra for short) of the Lusztig extension. Surprisingly, we can construct the KLR-type generators of Lusztig extensions by following Brundan and Kleshchev, and then fast give an isomorphism between direct sums of blocks of cyclotomic affine (and degenerate affine) Hecke algebras and cyclotomic BK subalgebras.

 In order to give an overview of our main results in this article, let $\mfD$ be the directed Dynkin diagram of a fixed Weyl group $\mcW$ with vertex set $[n]:=\{1,2,\cdots,n\}$. We use $(a_{rs})_{r,s\in [n]}$ to denote the Cartan matrix of $\mfD$:	
\begin{align*}a_{rs}:=\begin{cases}2 & \text{if $r=s$};\\
0& \text{if $r\nslash s$};\\
-1 & \text{if $r-\-s$, or $\xymatrix@C10pt{r\ar@{=}[r]|{\langle}\ar@{=}[r]&s}$, or  $\xymatrix@C10pt{r\ar@3{-}[r]|{\langle}\ar@3{-}[r]&s}$};\\
-2 & \text{if$ \xymatrix@C10pt{r\ar@{=}[r]|{\rangle}\ar@{=}[r]&s}$};\\
-3 & \text{if $\xymatrix@C10pt{r\ar@3{-}[r]|{\rangle}\ar@3{-}[r]&s}$}.
\end{cases}\end{align*}
Here the symbol $r \nslash s$ indicates that $r$ is not connected with $s$ by edges; $r -\- s$ indicates that $r$ is connected with $s$ by a single edge; $\xymatrix@C10pt{r\ar@{=}[r]|{\rangle}\ar@{=}[r]&s}$ indicates that $r$ is connected with $s$ by a double edge and there is an arrow from $r$ to $s$; and $\xymatrix@C10pt{r\ar@3{-}[r]|{\rangle}\ar@3{-}[r]&s}$ indicates that $r$ is connected with $s$ by a triple edge and there is an arrow from $r$ to $s$.

\vskip5pt
Let $I^n$ be the set of $n$-tuples of the abelian group $I=\mathbb{Z}/e\mathbb{Z}$ ($e\geq 0$ and $e\neq 1$). The Weyl group $\mcW$ acts on the left on $I^n$ via (\ref{group action}). We fix a $\mcW$-orbit $\mcC$ of $I^n$ and a system $\{\epsilon(\bi) \ | \ \bi\in \mcC\}$ of mutually orthogonal idempotents. Fix a ground field $\Bbbk$. We consider the {\itshape BK subalgebra} $\tilde{\mathcal{L}}$ generated by
$$\{y_1,\cdots,y_n,\psi_1,\cdots,\psi_n\}\cup \{\epsilon(\bi) \ | \ \bi\in \mcC\}\cup \{f^{-1} \ | \ f\in \Bbbk[y_1,\cdots,y_n] \ \mbox{with} \ f(0)\neq0\}$$
and some KLR-type basis relations (1)-(9) of Theorem \ref{thm:deKLR basis}.
\vskip5pt

Let $\mcH$ be the degenerate affine Hecke algebra generated by $x_1,\cdots,x_n,t_1,\cdots,t_n$ and relations (\ref{deaffhecke x})-(\ref{detttttt}). From now on, we fix $\Lambda=(\Lambda_i)_{i\in I}\in \mathbb{N}^I$ (we follow the convention that $\mathbb{N}=\{0,1,2,\cdots\}$) with $\sum_{i\in I}\Lambda_i<\infty$. Let $\tilde{\mathcal{L}}(\Lambda)=\tilde{\mathcal{L}}/\langle y_1^{\Lambda_{i_1}}\epsilon(\bi) | \bi\in \mcC \rangle \ \ \mbox{and}$ and $\mcH(\Lambda)=\mcH/\langle \prod_{i\in I}(x_1-i)^{\Lambda_i} \rangle$ be the corresponding cyclotomic quotients.
\vskip5pt

The following is our first main result in this article for the degenerate affine Hecke algebra:

\vskip5pt
\noindent{\bf Theorem \ref{thm:derho}} {\itshape There is an isomorphism of algebras $\tilde{\mathcal{L}}(\Lambda)\cong \mcH(\Lambda)e(\mcC).$}
\vskip5pt
Next, we give a similar construction for the non-degenerate affine Hecke algebra.  Fix $q\in\Bbbk$, $q\neq 0,1$. We consider the {\itshape BK subalgebra} $\tilde{\mathcal{L}}_q$ generated by
$$\{Y_1,\cdots,Y_n,\Psi_1,\cdots,\Psi_n\}\cup \{\epsilon(\bi) \ | \ \bi\in \mcC\}\cup \{f^{-1} \ | \ f\in \Bbbk[Y_1,\cdots,Y_n] \ \mbox{with} \ f(0)\neq0\}$$
and some KLR-type basis relations (1)-(10) of Theorem \ref{thm:KLR basis}.
\vskip5pt
Let $\mcH_q$ be the non-degenerate affine Hecke algebra generated by $X_1,\cdots,X_n, $ $T_1,\cdots,T_n$ and relations (\ref{affhecke X})-(\ref{TTTTTT}). Let $\tilde{\mathcal{L}}_q(\Lambda)=\tilde{\mathcal{L}}/\langle Y_1^{\Lambda_{i_1}}\epsilon(\bi) | \bi\in \mcC \rangle \ \ \mbox{and}$ and $\mcH_q(\Lambda)=\mcH/\langle \prod_{i\in I}(X_1-i)^{\Lambda_i} \rangle$ be the corresponding cyclotomic quotients.
\vskip5pt
The following is our second main result in this paper for the non-degenerate affine Hecke algebra:

\vskip5pt
\noindent{\bf Theorem \ref{thm:rho}} {\itshape There is an isomorphism of algebras $\tilde{\mathcal{L}}_q(\Lambda)\cong \mcH_q(\Lambda)e(\mcC).$}
\vskip5pt

 We now sketch the contents of the paper. In section 2, we give the group action of $\mcW$ on the $n$-tuples $I^n$ which will be used as an index set in the whole article. In Section 3, we first review the degenerate affine Hecke algebras and recall their Bernstein-Zelevinski basis. We then introduce the Lusztig extension of degenerate affine Hecke algebras and give their KLR-type generators. Our first main result Theorem \ref{thm:derho} in this paper are then proved after introducing the BK subalgebras. In Section 4, we consider the non-degenerate case in a similar way. The last section is aimed to give the general and unified definition for the KLR type algebras in the two cases above.

	\section{Preliminaries}
	\subsection{The divided difference operator of polynomial rings}  We view the Weyl group $\mcW$ as being generated as a Coxeter group
	with generators $\{\sigma_r, r\in[n]\}$. Let $\Bbbk[x]$ be the polynomial ring with indeterminates $\{x_r | r\in [n]\}$ over the field $\Bbbk$. There is an action of $\mcW$ from the left on $\Bbbk[x]$ (by the ring automorphism) by defining
	\begin{equation}\label{desr(xs)}
	\sigma_r(x_s)=x_s-a_{sr}x_r.
	\end{equation}
	for $r, s\in [n]$
	
Using the $\mcW$-action above, we define the {\it divided difference operators} $\parti_r$ on $\Bbbk[x]$ for all $r\in[n]$ as $$\parti_r(f)=\frac{\sigma_r(f)-f}{x_r}.$$   By straightforward calculations, the divided difference operators satisfy the Leibniz rule
\begin{equation*}\label{deLeibniz}\parti_r(fg)=\parti_r(f)g+\sigma_r(f)\parti_r(g),
\end{equation*}
	for $f,g\in \Bbbk[x]$ and $r\in [n]$, and the relations
	$$\sigma_r(\parti_r(f))=\parti_r(f), \quad \parti_r(\sigma_r(f))=-\parti_r(f)$$

 Let $\Bbbk(x)$ be the corresponding rational function field, then the $\mcW$-action on $\Bbbk[x]$ above can be extended to an action $w\colon \tfrac{f}{g}\mapsto \tfrac{w(f)}{w(g)}$ of $\mcW$ on $\Bbbk(x)$ (by the field automorphism). This means that the action of the divided difference operators on $\Bbbk[x]$ also extends to operators on $\Bbbk(x)$.
	
	Let $F$ be the quotient field of the subalgebra $$Z=\{f\in \Bbbk[x] | w(f)=f \ \mbox{for every}\ w\in \mcW\}$$ of $\mcW$-invariants. By \cite[3.12 (a)]{Lusztig89}, there is a natural $\Bbbk$-algebra isomorphism
	\begin{equation} \label{deK(x)} \Bbbk[x]\otimes_{Z} F\to \Bbbk(x), f\otimes g\mapsto fg. \end{equation}

	\subsection{The Demazure operator of Laurent polynomials rings}  Let $\Bbbk[X^{\pm 1}]$ be the Laurent polynomial ring in the indeterminates $\{X_r |  r\in [n]\}$. There is an action of $\mcW$ from the left on $\Bbbk[X^{\pm 1}]$ (by the ring automorphism) by defining \begin{equation}\label{nsr(yt)}
	\sigma_r(X_s)=X_sX_r^{-a_{sr}}.
	\end{equation}
	for all $r,s\in [n]$.
	Using the $\mcW$-action above, we define \emph{Demazure operators} $\mathrm{D}_r$ on $\Bbbk[X^{\pm1}]$ for all $r\in[n]$ as
	\begin{equation}\mathrm{D}_{r}(f):=\frac{\sigma_r(f)-f}{1-X_r}.
	\end{equation}
It is well-known that the Demazure operators on $\Bbbk[X^{\pm1}]$ also satisfy the Leibniz rule
\begin{equation*}\label{Leibniz}\mathrm{D}_r(fg)=\mathrm{D}_r(f)g+\sigma_r(f)\mathrm{D}_r(g).
\end{equation*}
for $f,g\in \Bbbk[X^{\pm 1}]$ and $r\in[n]$.
	
		 Let $\Bbbk(X)$ be the corresponding rational function field, then the action of $\mcW$ on $\Bbbk[X^{\pm 1}]$ extends to an action $w\colon \frac{f}{g}\mapsto \frac{w(f)}{w(g)}$ of $\mcW$ on $\Bbbk(X)$ (by the field automorphism). This means that the action of the Demazure operators on $\Bbbk[X^{\pm1}]$ also extends to operators on $\Bbbk(X)$.

	Let $\mcF$ be the quotient field of the subalgebra $$\mathcal{Z}=\{f\in \Bbbk[X^{\pm 1}] | w(f)=f \ \mbox{for every}\ w\in \mcW\}$$ of $\mcW$-invariants. Similar to (\ref{deK(x)}), there is a natural $\Bbbk$-algebra isomorphism
	\begin{equation} \label{K(X)} \Bbbk[X^{\pm 1}]\otimes_{\mcZ} \mcF\to \Bbbk(X), f\otimes g\mapsto fg. \end{equation}

	\subsection{The numbers game} Let $I$ be the abelian group $\mathbb{Z}/e\mathbb{Z}$ ($e\geq 0$ and $e\neq 1$). For $r\in [n]$ and $\bi=(i_s)_{s\in [n]}\in I^n$, we define the map
\begin{equation}\label{group action}\sigma_r(\bi)_s=i_s-a_{sr}i_r.
\end{equation}
Then these maps induce an action of $\mcW$ on $I^{n}$. In fact, this is a special case of the numbers game which applies a combinatorial model of the Coxeter groups \cite[Theorem 4.3.1(ii)]{Bjorner-Brenti}.
\begin{lemma} For $r,s\in [n]$, there holds that
		
	$(1)$ $\sigma_r^2=1$.
	
	$(2)$ $\sigma_r\sigma_s=\sigma_s\sigma_r$, if $r\nslash s$.
	
	$(3)$ $\sigma_r\sigma_s\sigma_r=\sigma_s\sigma_r\sigma_s$, if $r-\-s$.
	
	$(4)$ $(\sigma_r\sigma_s)^2=(\sigma_s\sigma_r)^2$, if $\xymatrix@C10pt{r\ar@{=}[r]|{\rangle}\ar@{=}[r]&s}$.
	
	$(5)$ $(\sigma_r\sigma_s)^3=(\sigma_s\sigma_r)^3$, if $\xymatrix@C10pt{r\ar@3{-}[r]|{\rangle}\ar@3{-}[r]&s}$.
\end{lemma}
\begin{proof} The assertions (1)-(2) can be proved directly from the construction of $\sigma_r$. The assertions (3)-(5) follow from the following commutative diagrams:
\[
\begin{tikzpicture}\draw[-](-6.5,2)--(-6.2,2);
\node[scale=0.6] at(-6.7,2){$i_r$};
\node[scale=0.6] at(-6,2){$i_s$};
\draw[->](-6.8,1.9)--(-7.1,1.5);
\draw[->](-5.9,1.9)--(-5.6,1.5);
\draw[-](-6.95,1.4)--(-7.25,1.4);
\node[scale=0.6] at(-7.6,1.4){$i_r+i_s$};
\node[scale=0.6] at(-6.75,1.4){$-i_s$};
\draw[-](-5.75,1.4)--(-5.45,1.4);
\node[scale=0.6] at(-5.95,1.4){$-i_r$};
\node[scale=0.6] at(-5.05,1.4){$i_r+i_s$};
\draw[->](-7.1,1.3)--(-7.1,.8);
\draw[->](-5.6,1.3)--(-5.6,.8);
\node[scale=0.6] at(-7.7,.7){$-i_r-i_s$};
\node[scale=0.6] at(-6.75,.7){$i_r$};
\node[scale=0.6] at(-5.95,.7){$i_s$};
\node[scale=0.6] at(-4.95,.7){$-i_r-i_s$};
\draw[-](-7.25,.7)--(-6.95,.7);
\draw[-](-5.75,.7)--(-5.45,.7);
\draw[-](-6.5,.2)--(-6.2,.2);
\node[scale=0.6] at(-6.8,.2){$-i_s$};
\node[scale=0.6] at(-5.9,.2){$-i_r$};
\draw[->](-7.,.6)--(-6.55,.3);
\draw[->](-5.7,.6)--(-6.15,.3);
\node[scale=0.5] at (-7.15,1.8){$\sigma_s$};
\node[scale=0.5] at (-5.6,1.8){$\sigma_r$};
\node[scale=0.5] at (-7.3,1){$\sigma_r$};
\node[scale=0.5] at (-5.4,1){$\sigma_s$};
\node[scale=0.5] at (-7,.4){$\sigma_s$};
\node[scale=0.5] at (-5.7,.4){$\sigma_r$};

\draw[-](-2,2.025)--(-1.7,2.025);
\draw[-](-2,1.98)--(-1.7,1.98);
\node[scale=0.4] at (-1.85,2){$\rangle$};
\node[scale=0.6] at(-2.2,2){$i_r$};
\node[scale=0.6] at(-1.5,2){$i_s$};

\draw[->](-2.3,1.9)--(-3.1,1.6);
\draw[->](-1.4,1.9)--(-.6,1.6);

\draw[-](-3.25,1.425)--(-2.95,1.425);
\draw[-](-3.25,1.38)--(-2.95,1.38);
\node[scale=0.4] at (-3.1,1.4){$\rangle$};
\node[scale=0.6] at(-3.7,1.4){$i_r+2i_s$};
\node[scale=0.6] at(-2.75,1.4){$-i_s$};

\draw[-](-.75,1.425)--(-.45,1.425);
\draw[-](-.75,1.38)--(-.45,1.38);
\node[scale=0.4] at (-.6,1.4){$\rangle$};
\node[scale=0.6] at(-.95,1.4){$-i_r$};
\node[scale=0.6] at(-.05,1.4){$i_r+i_s$};

\draw[->](-3.1,1.3)--(-3.1,.8);
\draw[->](-0.6,1.3)--(-0.6,.8);

\draw[-](-3.25,.725)--(-2.95,.725);
\draw[-](-3.25,.68)--(-2.95,.68);
\node[scale=0.4] at (-3.1,.7){$\rangle$};
\node[scale=0.6] at(-3.8,.7){$-i_r-2i_s$};
\node[scale=0.6] at(-2.5,.7){$i_r+i_s$};

\draw[-](-.75,.725)--(-.45,.725);
\draw[-](-.75,.68)--(-.45,.68);
\node[scale=0.4] at (-0.6,.7){$\rangle$};
\node[scale=0.6] at(-1.2,.7){$i_r+2i_s$};
\node[scale=0.6] at(0,.7){$-i_r-i_s$};

\draw[->](-3.1,.6)--(-3.1,.1);
\draw[->](-.6,.6)--(-.6,.1);

\draw[-](-3.25,.025)--(-2.95,.025);
\draw[-](-3.25,-.02)--(-2.95,-.02);
\node[scale=0.4] at (-3.1,0){$\rangle$};
\node[scale=0.6] at(-3.4,0){$i_r$};
\node[scale=0.6] at(-2.45,0){$-i_r-i_s$};

\draw[-](-.75,.025)--(-.45,.025);
\draw[-](-.75,-.02)--(-.45,-.02);
\node[scale=0.4] at (-.6,0){$\rangle$};
\node[scale=0.6] at(-1.25,0){$-i_r-2i_s$};
\node[scale=0.6] at(-.25,0){$i_s$};

\draw[-](-2,-.675)--(-1.7,-.675);
\draw[-](-2,-.72)--(-1.7,-.72);
\node[scale=0.4] at (-1.85,-.7){$\rangle$};
\node[scale=0.6] at(-2.2,-.7){$-i_r$};
\node[scale=0.6] at(-1.5,-.7){$-i_s$};

\draw[->](-3,-.15)--(-2.05,-.5);
\draw[->](-.7,-.15)--(-1.65,-.5);

\node[scale=0.5] at (-2.9,1.9){$\sigma_s$};
\node[scale=0.5] at (-.8,1.9){$\sigma_r$};
\node[scale=0.5] at (-3.3,1.05){$\sigma_r$};
\node[scale=0.5] at (-.4,1.05){$\sigma_s$};
\node[scale=0.5] at (-3.3,.35){$\sigma_s$};
\node[scale=0.5] at (-.4,.35){$\sigma_r$};
\node[scale=0.5] at (-2.8,-.45){$\sigma_r$};
\node[scale=0.5] at (-.8,-.45){$\sigma_s$};

\draw[-](3.2,2.035)--(3.5,2.035);
\draw[-](3.2,2)--(3.5,2);
\draw[-](3.2,1.965)--(3.5,1.965);
\node[scale=0.4] at (3.35,2){$\rangle$};
\node[scale=0.6] at(3,2){$i_r$};
\node[scale=0.6] at(3.7,2){$i_s$};

\draw[->](2.9,1.9)--(2,1.6);
\draw[->](3.8,1.9)--(4.6,1.6);

\draw[-](1.75,1.435)--(2.05,1.435);
\draw[-](1.75,1.4)--(2.05,1.4);
\draw[-](1.75,1.365)--(2.05,1.365);
\node[scale=0.4] at (1.9,1.4){$\rangle$};
\node[scale=0.6] at(1.3,1.4){$i_r+3i_s$};
\node[scale=0.6] at(2.3,1.4){$-i_s$};

\draw[-](4.45,1.435)--(4.75,1.435);
\draw[-](4.45,1.4)--(4.75,1.4);
\draw[-](4.45,1.365)--(4.75,1.365);
\node[scale=0.4] at (4.6,1.4){$\rangle$};
\node[scale=0.6] at(4.2,1.4){$-i_r$};
\node[scale=0.6] at(5.2,1.4){$i_r+i_s$};

\draw[->](1.9,1.3)--(1.9,.8);
\draw[->](4.6,1.3)--(4.6,.8);

\draw[-](1.75,.035)--(2.05,.035);
\draw[-](1.75,0)--(2.05,0);
\draw[-](1.75,-.035)--(2.05,-.035);
\node[scale=0.4] at (1.9,0){$\rangle$};
\node[scale=0.6] at(1.2,0){$2i_r+3i_s$};
\node[scale=0.6] at(2.6,0){$-i_r-2i_s$};

\draw[-](4.45,.035)--(4.75,.035);
\draw[-](4.45,0)--(4.75,0);
\draw[-](4.45,-.035)--(4.75,-.035);
\node[scale=0.4] at (4.6,0){$\rangle$};
\node[scale=0.6] at(3.85,0){$-2i_r-3i_s$};
\node[scale=0.6] at(5.2,0){$i_r+2i_s$};

\draw[->](1.9,.6)--(1.9,.1);
\draw[->](4.6,.6)--(4.6,.1);

\draw[-](1.75,.735)--(2.05,.735);
\draw[-](1.75,.7)--(2.05,.7);
\draw[-](1.75,.665)--(2.05,.665);
\node[scale=0.4] at (1.9,.7){$\rangle$};
\node[scale=0.6] at(1.2,.7){$-i_r-3i_s$};
\node[scale=0.6] at(2.5,.7){$i_r+2i_s$};

\draw[-](4.45,.735)--(4.75,.735);
\draw[-](4.45,.7)--(4.75,.7);
\draw[-](4.45,.665)--(4.75,.665);
\node[scale=0.4] at (4.6,.7){$\rangle$};
\node[scale=0.6] at(3.95,.7){$2i_r+3i_s$};
\node[scale=0.6] at(5.2,.7){$-i_r-i_s$};
\draw[->](1.9,-.1)--(1.9,-.6);
\draw[->](4.6,-.1)--(4.6,-.6);

\draw[-](1.75,-.675)--(2.05,-.675);
\draw[-](1.75,-.72)--(2.05,-.72);
\node[scale=0.4] at (1.9,-.7){$\rangle$};
\node[scale=0.6] at(1.2,-.7){$-2i_r-3i_s$};
\node[scale=0.6] at(2.55,-.7){$i_r+i_s$};

\draw[-](4.45,-.675)--(4.75,-.675);
\draw[-](4.45,-.72)--(4.75,-.72);
\node[scale=0.4] at (4.6,-.7){$\rangle$};
\node[scale=0.6] at(3.95,-.7){$i_r+3i_s$};
\node[scale=0.6] at(5.3,-.7){$-i_r-2i_s$};

\draw[-](1.75,-1.375)--(2.05,-1.375);
\draw[-](1.75,-1.42)--(2.05,-1.42);
\node[scale=0.4] at (1.9,-1.4){$\rangle$};
\node[scale=0.6] at(1.5,-1.4){$i_r$};
\node[scale=0.6] at(2.65,-1.4){$-i_r-i_s$};

\draw[-](4.45,-1.375)--(4.75,-1.375);
\draw[-](4.45,-1.42)--(4.75,-1.42);
\node[scale=0.4] at (4.6,-1.4){$\rangle$};
\node[scale=0.6] at(3.85,-1.4){$-i_r-3i_s$};
\node[scale=0.6] at(5,-1.4){$i_s$};
\draw[->](1.9,-.8)--(1.9,-1.3);
\draw[->](4.6,-.8)--(4.6,-1.3);

\draw[-](3.2,-2.135)--(3.5,-2.135);
\draw[-](3.2,-2.1)--(3.5,-2.1);
\draw[-](3.2,-2.065)--(3.5,-2.065);
\node[scale=0.4] at (3.35,-2.1){$\rangle$};
\node[scale=0.6] at(3,-2.1){$-i_r$};
\node[scale=0.6] at(3.7,-2.1){$-i_s$};

\draw[->](2,-1.55)--(3.1,-1.9);
\draw[->](4.5,-1.55)--(3.55,-1.9);

\node[scale=0.5] at (2.1,1.9){$\sigma_s$};
\node[scale=0.5] at (4.4,1.9){$\sigma_r$};
\node[scale=0.5] at (1.7,1.05){$\sigma_r$};
\node[scale=0.5] at (4.8,1.05){$\sigma_s$};
\node[scale=0.5] at (1.7,.35){$\sigma_s$};
\node[scale=0.5] at (4.8,.35){$\sigma_r$};
\node[scale=0.5] at (1.7,-.35){$\sigma_r$};
\node[scale=0.5] at (4.8,-.35){$\sigma_s$};
\node[scale=0.5] at (4.8,-1.05){$\sigma_r$};
\node[scale=0.5] at (1.7,-1.05){$\sigma_s$};
\end{tikzpicture}
\]
\end{proof}

In the sequel, for simplicity, we will write $i_{s1}:=\sigma_r(\bi)_s, i_{s2}:=\sigma_r\sigma_s(\bi)_s, i_{s3}:=\sigma_r\sigma_s\sigma_r(\bi)_s, i_{s4}:=\sigma_r\sigma_s\sigma_r\sigma_s(\bi)_s, i_{s5}:=\sigma_r\sigma_s\sigma_r\sigma_s\sigma_r(\bi)_s$.

\section{The BK subalgebra in degenerate case}

\subsection{Degenerate affine Hecke algebras and their BZ basis}
In \cite{Drinfeld}, Drinfeld introduced a machinery to define degenerate affine Hecke algebras and gave the concise form of type A. We notice that Drinfeld's degenerate affine Hecke algebra is also a quotient of the graded Hecke algebra induced by Lusztig in \cite[Proposition 4,4]{Lusztig89}. Following Drinfeld, the {\itshape degenerate affine Hecke algebra}  $\mcH$ of $\mcW$ is defined to be the associated unital $\Bbbk$-algebra with generators $\{x_r, t_r  |  r\in [n]\}$ subject to the following relations for all admissible indices:
\begin{align}
\label{deaffhecke x}
x_rx_s&=x_sx_r;
\\
\label{detrxs}t_r x_s &= \sigma_r(x_s)t_r+\parti_r(x_s);\\
\label{det^2}
\quad t_r^2 &= 1;
\\
\label{dett}
t_r  t_s &= t_s  t_r \hspace{43.5mm} \text{if $r \nslash s$};
\\
\label{dettt}
t_r t_st_r &= t_st_rt_s \hspace{39.5mm} \text{if $r -\- s$};
\\
\label{detttt}
(t_r t_s)^2 &= (t_st_r)^2 \hspace{39.1mm} \text{if $\xymatrix@C10pt{r\ar@{=}[r]|{\rangle}\ar@{=}[r]&s}$};\\
\label{detttttt}
(t_r t_s)^3 &= (t_st_r)^3 \hspace{39.5mm} \text{if $\xymatrix@C10pt{r\ar@3{-}[r]|{\rangle}\ar@3{-}[r]&s}$}.
\end{align}

For $w=\sigma_{r_1} \sigma_{r_2}\cdots \sigma_{r_m}\in \mcW$ a reduced expression we put $T_w:=T_{r_1}T_{r_2}\cdots T_{r_m}$. Then $T_w$ is a well-defined element in $\mcH$ and the algebra $\mcH$ is a free $\Bbbk[x]$-module with basis $\{T_w\ | \ w\in S_n\}$ (the Bernstein-Zelevinski basis) \cite[Lemma 3.4]{Lusztig89}.

By \cite[Proposition 3.11]{Lusztig89}, the center of $\mcH$ is $Z$. Then $\mcH$ can be seen as a $Z$-subalgebra (identified with the subspace $\mcH\otimes 1$) of the $Z$-algebra $$\mcH_F:=\mcH\otimes_Z F.$$ For any $s\in [n]$ and $f\in \Bbbk(x)$, as a consequence of \cite[3.12 (d)]{Lusztig89}, we get
\begin{equation} \label{deft} t_sf=\sigma_s(f)t_s+\parti_s(f).\end{equation}
Using the BZ basis of $\mcH$, there are decompositions \cite[3.12(c)]{Lusztig89}:
\begin{equation}\label{deTwdecom}\mcH_{F}=\oplus_{w\in \mcW}t_w\Bbbk(x)=\oplus_{w\in \mcW}\Bbbk(x)t_w. \end{equation}
So the rationalization algebra $\mcH_F$ is a free $\Bbbk(x)$-module of finite rank with basis $\{t_w \ | \ w\in \mcW\}$.

\subsection{Intertwining elements}
 For $r\in [n]$, we define the {\itshape intertwining elements} $\kappa_r$ in $\mcH_F$ as follows:
$$\phi_r:=t_r+x_r^{-1}.$$
The following result is important for us.
\begin{proposition} \label{prop:deintert basis}The algebra $\mcH_F$ is generated by $\{x_1,\cdots,x_n,\phi_1,\cdots,\phi_n,f^{-1}\ | \ 0\neq f\in K[x]\}$ subject to the following relations for all admissible $r,s$:
	\begin{align} \label{dexrxs}x_rx_s&=x_sx_r;\\
	\label{deff^{-1}} ff^{-1}&=f^{-1}f=1 \hspace{15mm} \text{$\forall 0\neq f\in \Bbbk[x]$};\\
\label{dephig}	\phi_rx_s&=\sigma_r(x_s)\phi_r;\\
	\label{dephi^2}\phi_r^2&=1-x_r^{-2};\\
	\label{dephi2}\phi_r\phi_s&=\phi_s\phi_r \hspace{20mm} \text{if $r\nslash s$};\\
	\label{dephi3}\phi_r\phi_s\phi_r&=\phi_s\phi_r\phi_s \hspace{16.4mm} \text{if $r-\-s$};\\
	\label{dephi4}(\phi_r\phi_s)^2&=(\phi_s\phi_r)^2\hspace{15mm} \text{if $\xymatrix@C10pt{r\ar@{=}[r]|{\rangle}\ar@{=}[r]&s}$};\\
	\label{dephi6}(\phi_r\phi_s)^3&=(\phi_s\phi_r)^3 \hspace{15mm} \text{if $\xymatrix@C10pt{r\ar@3{-}[r]|{\rangle}\ar@3{-}[r]&s}$}.
	\end{align}	
 \end{proposition}
\begin{proof} Since $t_r=\phi_r-x_r^{-1}$, thus the elements $\{x_1,\cdots,x_n,\phi_1,\cdots,\phi_n,f^{-1}\ | \ 0\neq f\in \Bbbk[x]\}$ are generators of $\mcH_F$ by (\ref{deTwdecom}). The statements (\ref{dexrxs}) and (\ref{deff^{-1}}) follow using the fact that $\Bbbk(x)$ is a commutative subalgebra of $\mcH_F$.  By \cite[Proposition 5.2]{Lusztig89}, the relations (\ref{dephig})-(\ref{dephi6}) hold. Next we prove the generating set is complete, that is, these relations generate all the relations.
	
	Assume $\mathcal{A}$ is the $\Bbbk$-subalgebra generated by $\{x_1,\cdots,x_n,\phi_1,\cdots,\phi_n,f^{-1}\ | \ 0\neq f\in \Bbbk[x]\}$ subject to the relations (\ref{dexrxs})-(\ref{dephi6}). Then there is an obvious surjective $\Bbbk$-algebra homomorphism
	$$\pi\colon \mathcal{A}\to \mcH_F, \ a\mapsto a.$$
This maps is $\Bbbk(x)$-linear. So $\pi$ is in fact a homomorphism of $\Bbbk(x)$-modules. Since $\mcH_F$ is a free $\Bbbk(x)$-module of rank $|\mcW|$ with basis $\{t_w \ | \ w\in \mcW\}$, so it is enough to show that $\mathcal{A}$ has also rank $|\mcW|$ as a $\Bbbk(x)$-module. In fact, we can define $\phi_w:=\phi_{r_1}\phi_{r_2}\cdots \phi_{r_m}$ if $w=\sigma_{r_1} \sigma_{r_2}\cdots \sigma_{r_m}$ is a reduced expression in $\mcW$. It is a well-defined element by the braid relations (\ref{dephi2})-(\ref{dephi6}). We claim that $\{\phi_w\ | \ w\in \mcW\}$ is a $\Bbbk(x)$-basis of $\mathcal{A}$. In fact, by the relation (\ref{dephig}), we can express any element $a\in\mathcal{A}$ in the form
$$a=\sum_{w\in \mcW}f_{w}\phi_w$$
where $f_w\in \Bbbk(x)$. Since $\phi_w=t_w+\sum_{u<w}f_ut_u$
for some rational functions $f_u\in \Bbbk(x)$, and $\{t_w \ | \ w\in W\}$ are $\Bbbk(x)$-linearly independent in $\mcH_F$, therefore $\{\phi_w \ | \ w\in \mcW\}$ are also $\Bbbk(x)$-linearly independent. \end{proof}
\noindent The Proposition above shows that $\mcH_F$ has decompositions
 \begin{equation} \label{de rationalization phi_w}
\mcH_F=\oplus_{w\in \mcW}\Bbbk(x)\phi_w=\oplus_{w\in \mcW}\phi_w\Bbbk(x).
\end{equation}

\subsection{The Lusztig extension of $\mcH$}
Let $\mathcal{E}$ be the unital $\Bbbk$-algebra with basis $\{\epsilon(\bi) | \bi\in \mcC\}$ which satisfy
\begin{equation} \label{deepsilon} \epsilon(\bi)\epsilon(\bj)=\delta^\bi_\bj\epsilon(\bi), \quad \sum_{\bi\in \mcC}\epsilon(\bi)=1.\end{equation}

The {\it Lusztig extension} of $\mcH$ with respect to $\mathcal{E}$ is the $\Bbbk$-algebra $\mathcal{L}$ which is equal as $\Bbbk$-space to the tensor product $$\mathcal{L}:=\mcH_F\otimes_\Bbbk\mathcal{E}=\oplus_{w\in \mcW,\bi\in \mcC}\phi_w\Bbbk(x)\epsilon(\bi)$$
of the rationalization algebra $\mcH_F$ and the semi-simple algebra $\mathcal{E}$. Multiplication is defined so that $\mcH_{F}$ (identified with the subspace $\mcH_{F}\otimes 1$) and $\mathcal{E}$ (identified with the subspace $1\otimes \mathcal{E}$) are subalgebras of $\mathcal{L}$, and in addition
 \begin{align}
\label{dexepis}x_r\epsilon(\bi)&=\epsilon(\bi)x_r,\\
\label{dephiepsi}\phi_r\epsilon(\bi)&=\epsilon(\sigma_r(\bi))\phi_r.
\end{align}
It is easy to see that \begin{equation}\label{tepsi}t_r\epsilon(\bi)=\epsilon(\sigma_r(\bi))t_r+x_r^{-1}\epsilon(\sigma_r(\bi))-x_r^{-1}\epsilon(\bi).\end{equation}

We notice that relations (\ref{deepsilon})-(\ref{tepsi}) have been implicitly indicated in \cite[Subsections 8.7-8.8]{Lusztig89}. Therefore, the reason for using the notion of Lusztig extension is clear.

By Proposition \ref{prop:deintert basis}, the Lusztig extension $\mathcal{L}$ has two generating sets, one is
$$\{x_1,\cdots,x_n,\phi_1,\cdots,\phi_n,f^{-1}, \epsilon(\bi)\ | \ 0\neq f\in \Bbbk[x], \bi\in \mcC\},$$
the other one is
$$\{x_1,\cdots,x_n,t_1,\cdots,t_n,f^{-1}, \epsilon(\bi)\ | \ 0\neq f\in \Bbbk[x], \bi\in \mcC\}.$$
 Moreover, the algebra $\mathcal{L}$ has decompositions
\begin{equation} \label{de Lusztig Tw}
\mathcal{L}=\oplus_{w\in \mcW,\bi\in \mcC}\Bbbk(x)t_w\epsilon(\bi)=\oplus_{w\in \mcW,\bi\in \mcC}t_w\Bbbk(x)\epsilon(\bi)
\end{equation}
and \begin{equation} \label{de Lusztig phi_w}
\mathcal{L}=\oplus_{w\in \mcW,\bi\in \mcC}\Bbbk(x)\phi_w\epsilon(\bi)=\oplus_{w\in \mcW,\bi\in \mcC}\phi_w\Bbbk(x)\epsilon(\bi).
\end{equation}

\subsection{Brundan-Kleshchev auxiliary elements}
 Recall that \cite[(3.12)]{Brundan-Kleshchev09}, for each $r\in [n]$, Brundan and Kleshchev introduced the element
\begin{equation} \label{yx} y_r:=\sum_{\bi\in \mcC}(x_r-i_r)\epsilon(\bi).\end{equation} Then $y_r\epsilon(\bi)=\epsilon(\bi)y_r$ by \ref{dexepis}, and $y_r$ is a unit in $\mathcal{L}$ with $y_r^{-1}=\sum_{\bi\in \mcC}(x_r-i_r)^{-1}\epsilon(\bi).$
Let $\Bbbk[y]$ be the polynomial ring with $y:=y_1,y_2,\cdots,y_n$ and $\Bbbk(y)$ be the rational function field.
\begin{lemma} \label{lem:K(x)=K(y)} $\Bbbk(x)\otimes_\Bbbk\mathcal{E}=\Bbbk(y)\otimes_\Bbbk\mathcal{E}$ in $\mathcal{L}$ as subalgebras.
\end{lemma}
\begin{proof} If $g(y)$ is a polynomial in $\Bbbk[y]$, then $$g(y)=\sum_{\bi \in \mcC}g(y)\epsilon(\bi)=\sum_{\bi \in \mcC}g(x_1-i_1,\cdots,x_n-i_n)\epsilon(\bi)$$
	in $\mathcal{L}$. Therefore, if $g(y)\neq 0$, then $g(x_1-i_1,\cdots,x_n-i_n)\neq 0$ in $K[x]$ for all $\bi\in \mcC$. Thus $g(y)^{-1}=\sum_{\bi \in \mcC}g(x_1-i_1,\cdots,x_n-i_n)^{-1}\epsilon(\bi)$ exists in $\mathcal{L}$. So $\Bbbk(y)\otimes_\Bbbk\mathcal{E}$ is a subalgebra of $\mathcal{L}$. Combining (\ref{yx}) with $x_r=\sum_{\bi\in \mcC}(y_r+i_r)\epsilon(\bi)$, we know that $\Bbbk(x)\otimes_\Bbbk\mathcal{E}=\Bbbk(y)\otimes_\Bbbk\mathcal{E}$ in $\mathcal{L}$.
\end{proof}
Following Brundan and Kleshchev, for each $r\in [n]$, we define the element
\begin{equation*}\label{deqr(i)} q_r(\bi)=\begin{cases} 1-y_r & \text{if $i_r=0$},\\
1& \text{if $i_r=1, e\neq 2$},\\
(y_r-2)(1-y_r)^{-2}& \text{if $i_r=-1, e\neq 2$},\\
(1-y_r)^{-1}& \text{if $i_r=1, e=2$},\\
1-(y_r+i_r)^{-1}& \text{if $i_r\neq 0, \pm 1$.}
\end{cases}
\end{equation*}
Notice that the Weyl group $\mcW$ acts on the left on $\Bbbk[y]$ since
$$\sigma_r(y_s)=\sum_{\bi\in \mcC}(\sigma_r(X_s)-\bi_s)\epsilon(\sigma_r(\bi))=y_{s}-a_{sr}y_r,$$
for all $r,s\in [n]$. For simplicity, we shall write $y_{r1}:=\sigma_s(y_r), y_{r2}:=\sigma_r\sigma_s(y_r), y_{r3}=\sigma_s\sigma_r\sigma_s(y_r), y_{r4}=\sigma_r\sigma_s\sigma_r\sigma_s(y_r), y_{r5}=\sigma_s\sigma_r\sigma_s\sigma_r\sigma_s(y_r)$ in the sequel.

Similar to the proof of \cite[(3.27)-(3.29)]{Brundan-Kleshchev09}, we obtain
\begin{lemma} Let $r,s\in [n]$ and $\bi\in \mcC$. Then
	\begin{equation}\label{deqrsrqrsr} q_r(\bi)\sigma_r(q_r(\sigma_r(\bi)))=\begin{cases}1-y_r^2 & \text{if $i_r=0$};\\
	-(y_r+2)(1+y_r)^{-2}&\text{if $i_r=1, e\neq 2$};\\
	(y_r-2)(1-y_r)^{-2}&\text{if $i_r=-1, e\neq2$};\\
	(1-y_r)^{-2}&\text{if $i_r=1, e=2$};\\
	1-(y_r+i_r)^{-2}& \text{if $i_r\neq 0, \pm 1$}.
	\end{cases}
	\end{equation}
\end{lemma}
Inside $\mathcal{L}$, for each $r\in [n]$, we set $$\theta_r:=\phi_r\sum_{\bi\in \mcC}q_r^{-1}(\bi)\epsilon(\bi).$$

Using Lemma \ref{lem:K(x)=K(y)} and relations (\ref{dexrxs})-(\ref{dephi6}), we know that the elements $\theta_r$ have the following nice properties:

\begin{align}\label{dethetaepsilon} \theta_r\epsilon(\bi)&=\epsilon(\sigma_r(\bi))\theta_r;\\
\label{deftheta} y_s\theta_r&=\theta_r\sigma_r(y_s);\\
\label{detheta2}\theta_r\theta_s&=\theta_s\theta_r \hspace{19mm} \text{if $r\nslash s$};\\
\label{detheta3}\theta_r\theta_s\theta_r&=\theta_s\theta_r\theta_s \hspace{16mm} \text{if $r-\-s$};\\
\label{detheta4}(\theta_r\theta_s)^2&=(\theta_s\theta_r)^2\hspace{15mm} \text{if $\xymatrix@C10pt{r\ar@{=}[r]|{\rangle}\ar@{=}[r]&s}$};\\
\label{detheta6}(\theta_r\theta_s)^3&=(\theta_s\theta_r)^3 \hspace{15mm} \text{if $\xymatrix@C10pt{r\ar@3{-}[r]|{\rangle}\ar@3{-}[r]&s}$};
\end{align}
and \begin{equation}\label{detheta^2} \theta_r^2\epsilon(\bi)=\begin{cases}-y_r^{-2}\epsilon(\bi) & \text{if $i_r=0$},\\
	-y_r\epsilon(\bi) &\text{if $i_r=1, e\neq 2$},\\
	y_r\epsilon(\bi) & \text{if $i_r=-1, e\neq 2$},\\
	-y_r^2\epsilon(\bi)& \text{if $i_r=1, e=2$},\\
	\epsilon(\bi)& \text{if $i_r\neq0,\pm1$}.
	\end{cases}
\end{equation}

	\begin{remark} For a reduced expression $w=\sigma_{r_1} \sigma_{r_2}\cdots \sigma_{r_m}$ in $\mcW$, we define the element $$\theta_w:=\theta_{r_1}\theta_{r_2}\cdots \theta_{r_m}$$
	 in $\mathcal{L}$. Then it is a well-defined element by the braid relations. Since $\phi_r=\theta_r\sum_{\bi\in \mcC} q_r(\bi)\epsilon(\bi)$, thus by Lemma \ref{lem:K(x)=K(y)} and the decompositions (\ref{de Lusztig phi_w}), it can be proved that
	    \begin{equation} \label{de Lusztig theta_w}
	    \mathcal{L}=\oplus_{w\in \mcW,\bi\in \mcC}\Bbbk(y)\theta_w\epsilon(\bi)=\oplus_{w\in \mcW,\bi\in \mcC}\theta_w\Bbbk(y)\epsilon(\bi).
	    \end{equation}
	 \end{remark}

\subsection{The KLR-type generators of $\mathcal{L}$} In the Lusztig extension $\mathcal{L}$, for each $r\in [n]$, we define an element \begin{equation} \label{depsitheta}\psi_r=\sum_{\bi\in \mcC}[\theta_r-\delta_{i_r}^{0}y_r^{-1}]\epsilon(\bi),
\end{equation}
where $$\delta_{i_r}^0=\begin{cases}
1 & \text{if $i_r=0$},\\
0 & \text{if $i_r\neq 0$}
\end{cases}$$ is the Kronecker delta.

Using these elements, we can construct a KLR-type generators of $\mathcal{L}$.
\begin{theorem} \label{thm:deKLR basis} The algebra $\mathcal{L}$ is generated by $\{y_1,\cdots,y_n,\psi_1,\cdots,\psi_n,f^{-1}, \epsilon(\bi)\ | \ 0\neq f\in \Bbbk[y], \bi\in \mcC\}$ subject to the following relations for all admissible indices:
	
	$(1)$  $\epsilon(\bi)\epsilon(\bj)=\delta_{\bi}^{\bj}\epsilon(\bi), \quad \sum_{\bi\in \mcC}\epsilon(\bi)=1.$
	
	$(2)$  $y_r\epsilon(\bi)=\epsilon(\bi)y_r, y_ry_s=y_sy_r.$
	
	$(3)$ For any $0\neq f\in \Bbbk[y]$, there holds that $ff^{-1}=f^{-1}f=1.$

	$(4)$ $\psi_r\epsilon(\bi)=\epsilon(\sigma_r(\bi))\psi_r.$
	
	$(5)$  $\psi_ry_s\epsilon(\bi)=(\sigma_r(y_s)\psi_r+\delta_{i_r}^{0}\parti_r(y_s))\epsilon(\bi).$
	
	$(6)$
	$\psi_r^2\epsilon(\bi)=\begin{cases}0&\text{if $i_r=0$};\\
	-y_r\epsilon(\bi) & \text{if $i_r=1, e\neq 2$};\\
	y_r\epsilon(\bi) & \text{if $i_r=-1, e\neq 2$};\\
	-y_r^2\epsilon(\bi) & \text{if $i_r=1, e=2$};\\
	\epsilon(\bi)&\text{if $i_r\neq0,\pm 1$}.
	\end{cases}$

	$(7)$  If $r\nslash s$, then
$\psi_r\psi_s=\psi_s\psi_r.$
	
	$(8)$ If $r-\-s$, then 	\begin{equation*}\label{depsi3}
	[\psi_r\psi_s\psi_r-\psi_s\psi_r\psi_s]\epsilon(\bi)=\begin{cases}\epsilon(\bi) & \text{if $i_r=-i_s=1, e\neq 2$},\\
	-\epsilon(\bi) &\text{if $i_r=-i_s=-1, e\neq 2$},\\
	(y_r-y_s)\epsilon(\bi) & \text{if $i_r=i_s=1, e=2$},\\
	0 & \text{else}.
	\end{cases}
	\end{equation*}
	
	$(9)$ If $\xymatrix@C10pt{r\ar@{=}[r]|{\rangle}\ar@{=}[r]&s}$, then
	\begin{equation*}\label{depsi4} [(\psi_r\psi_s)^2-(\psi_s\psi_r)^2]\epsilon(\bi)=\begin{cases}-\psi_r\epsilon(\bi) & \text{if $i_s=1, i_r=-2, e\neq 2$};\\
	\psi_r\epsilon(\bi)&\text{if $i_s=-1, i_r=2,e\neq 2$}; \\
	(y_r\psi_r+1)\epsilon(\bi) & \text{if $i_r=0,i_s=1,e=2$},\\
	2\psi_s\epsilon(\bi) & \text{if $i_r=-i_s=1,e\neq 2$};\\
	-2\psi_s \epsilon(\bi)&\text{if $i_r=-i_s=-1, e\neq 2$}; \\
	-4y_s\psi_s\epsilon(\bi) & \text{if $i_r=i_s=1,e=2$};\\
	0 & \text{else}.
	\end{cases}
	\end{equation*}
	
	$(10)$ If $\xymatrix@C10pt{r\ar@3{-}[r]|{\rangle}\ar@3{-}[r]&s}$, then
	\begin{align*}\label{depsi6}
	&[(\psi_r\psi_s)^3-(\psi_s\psi_r)^3]\epsilon(\bi)\notag\\
	&=\begin{cases}-\psi_r\psi_s\psi_r\epsilon(\bi) & \text{if $i_s=1,i_{r}=-3,e\neq 2,3$},\\
	\psi_r\psi_s\psi_r\epsilon(\bi) & \text{if $i_s=-1,i_r=3, e\neq 2,3$},\\
	3\psi_s\psi_r\psi_s\epsilon(\bi) & \text{if $i_r=-i_s=1,e\neq 2$},\\
	-3\psi_s\psi_r\psi_s\epsilon(\bi) & \text{if $i_r=-i_s=-1,e\neq 2$},\\
	-2y_r\psi_r\epsilon(\bi) & \text{if $i_s=\pm 1,i_{r}=\mp 2,e\neq 2$},\\
	-4(y_{r}^2+3y_sy_{s1})(y_r\psi_r+1)\epsilon(\bi) & \text{if $i_r=0,i_s=1, e=2$},\\
	-2y_s\psi_s\epsilon(\bi) & \text{if $i_r=\pm 1,i_{s}= 2, e=4$},\\
	\psi_s\epsilon(\bi) & \text{if $i_r=1,i_{s}\neq -2, e\neq 2$},\\
	-\psi_s\epsilon(\bi) & \text{if $i_r=-1,i_{s}\neq -2, e\neq 2$},\\
	\psi_s\epsilon(\bi) & \text{if $i_r=3,i_{s}=-2, e\neq 2,3$},\\
	-\psi_s\epsilon(\bi) & \text{if $i_r=-3,i_{s}=2, e\neq 2,3$},\\
	4(y_{r}^2+3y_sy_{s1})(y_s\psi_s+1)\epsilon(\bi) & \text{if $i_r=1, i_s=0, e=2$},\\
	(y_{s1}\psi_r\psi_s\psi_r-3y_{r1}\psi_s\psi_r\psi_s-y_r^2+3y_s^2)\epsilon(\bi) & \text{if $i_r=i_s=1, e=2$},\\
	-\psi_r\psi_s\psi_r\epsilon(\bi) & \text{if $i_r=0,i_s=1, e=3$},\\
	\psi_r\psi_s\psi_r\epsilon(\bi) & \text{if $i_r=0,i_s=-1, e=3$},\\
	0 & \text{else}.
	\end{cases}
	\end{align*}	
\end{theorem}
\begin{proof} The statement (1) holds since the definition of the semi-simple algebra $\mathcal{E}$, and the statements (2)-(3) follow straightforwardly from the construction of the elements $y_r$.
	
(4)	By (\ref{dethetaepsilon}) and (\ref{depsitheta}), we get \begin{equation}\label{depsiepsilon}\psi_r\epsilon(\bi)=(\theta_r-\delta_{i_r}^0y_r^{-1})\epsilon(\bi)=\epsilon(\sigma_r(\bi))(\theta_r-\delta_{i_r}^0y_r^{-1})=\epsilon(\sigma_r(\bi))\psi_r\end{equation}
	where the third identity holds since $\sigma_r(\bi)=\bi$ whenever $i_r=0$.
	
(5) As a result of (\ref{depsitheta}) and (\ref{deftheta}), it follows that
\begin{align*}\psi_ry_s\epsilon(\bi)&=(\theta_r-\delta_{i_r}^0y_r^{-1})y_s\epsilon(\bi)\\
&=(\sigma_r(y_s)\theta_r-\delta_{i_r}^0y_sy_r^{-1})\epsilon(\bi) \\
&=(\sigma_r(y_s)\psi_r+\delta_{i_r}^0\sigma_r(y_s)y_r^{-1}-\delta_{i_r}^0y_sy_r^{-1})\epsilon(\bi)\\
&=(\sigma_r(y_s)\psi_r+\delta_{i_r}^0\parti_r(y_s))\epsilon(\bi).
\end{align*}

(6) Using (\ref{depsitheta}), (\ref{dethetaepsilon}) and (\ref{detheta^2}), we obtain
\begin{align*}
\psi_r^2\epsilon(\bi)=\psi_r\epsilon(\sigma_r(\bi))\psi_r\epsilon(\bi)
&=(\theta_r-\delta_{i_r}^0y_r^{-1})(\theta_r-\delta_{i_r}^0y_r^{-1})\epsilon(\bi)\\
&=(\theta_r^2+\delta_{i_r}^0y^{-2}_r)\epsilon(\bi)\\
&=\begin{cases}0 & \text{if $i_r=0$},\\		
-y_r\epsilon(\bi) &\text{if $i_r=1, e\neq 2$},\\
y_r\epsilon(\bi) & \text{if $i_r=-1, e\neq 2$},\\
-y_r^2\epsilon(\bi)& \text{if $i_r=1, e=2$},\\
\epsilon(\bi)& \text{if $i_r\neq0,\pm1$}.
\end{cases}
\end{align*}

(7) Since $r\nslash s$, there holds that
\begin{align*}\psi_r\psi_s\epsilon(\bi)&=\psi_r\epsilon(\sigma_s(\bi))\psi_s\epsilon(\bi)\\
&=(\theta_r-\delta_{i_r}^0y_r^{-1})(\theta_s-\delta_{i_s}^0y_s^{-1})\epsilon(\bi)\\
&=(\theta_r\theta_s-\delta_{i_r}^0y_r^{-1}\theta_s-\delta_{i_s}^0\theta_ry_s^{-1}+\delta_{i_r}^0\delta_{i_s}^0y_r^{-1}y_s^{-1})\epsilon(\bi)\\
&=(\theta_s\theta_r-\delta_{i_r}^0\theta_sy_r^{-1}-\delta_{i_s}^0y_s^{-1}\theta_r+\delta_{i_r}^0\delta_{i_s}^0y_r^{-1}y_s^{-1})\epsilon(\bi)\\
&=(\theta_s-\delta_{i_s}^0y_s^{-1})(\theta_r-\delta_{i_r}^0y_r^{-1})\epsilon(\bi)	\\
&=\psi_s\psi_r\epsilon(\bi)
\end{align*}
Hence $\psi_r\psi_s=\psi_s\psi_r$.

(8) Since $r-\-s$, we have
\begin{align*}\psi_r\psi_s\psi_r\epsilon(\bi)&=(\theta_r-\delta_{i_{r2}}^0y_r^{-1})(\theta_s-\delta_{i_{s1}}^0y_s^{-1})(\theta_r-\delta_{i_r}^0y_r^{-1})\epsilon(\bi)\\
&=[\theta_r\theta_s\theta_r-\delta_{i_r}^0y_{s}^{-1}\theta_r\theta_s-\delta_{i_{s}}^{0}y_r^{-1}\theta_s\theta_r+\delta_{i_r}^0\delta_{i_{s}}^{0}(y_s^{-1}y_{s1}^{-1}\theta_r+y_r^{-1}y^{-1}_{r1}\theta_s)\\
&-\delta_{i_{s1}}^0y_{s1}^{-1}\theta_r^2-\delta_{i_r}^0\delta_{i_{s1}}^{0}y_{r}^{-2}y_s^{-1}]\epsilon(\bi)
\end{align*}
Similarly, we can show that
\begin{align*}\psi_s\psi_r\psi_s\epsilon(\bi)&=(\theta_s-\delta_{i_{s2}}^0y_s^{-1})(\theta_r-\delta_{i_{r1}}^0y_r^{-1})(\theta_s-\delta_{i_s}^0y_s^{-1})\epsilon(\bi)\\
&=[\theta_s\theta_r\theta_s-\delta_{i_{r}}^{0}y_s^{-1}\theta_r\theta_s-\delta_{i_s}^0y_{r}^{-1}\theta_s\theta_r+\delta_{i_{r}}^{0}\delta_{i_{s}}^0(y_s^{-1}y^{-1}_{s1}\theta_r+y_r^{-1}y_{r1}^{-1}\theta_s)\\
&-\delta_{i_{r1}}^0y_{r1}^{-1}\theta_s^2-\delta_{i_{r}}^{0}\delta_{i_{r1}}^0y_r^{-1}y_{s}^{-2}]\epsilon(\bi).
\end{align*}
By (\ref{detheta3}) and (\ref{detheta^2}), we arrive at
\begin{align*}
[\psi_r\psi_s\psi_r-\psi_s\psi_r\psi_s]\epsilon(\bi)&=\delta_{i_{r1}}^0[y_{r1}^{-1}(\theta_s^2-\theta_r^2)+\delta_{i_r}^0(y_r^{-1}y_{s}^{-2}-y_{r}^{-2}y_{s}^{-1})]\epsilon(\bi)\\
&=\begin{cases}\epsilon(\bi) & \text{if $i_r=-i_s=1, e\neq 2$},\\
-\epsilon(\bi) &\text{if $i_r=-i_s=-1, e\neq 2$},\\
(y_r-y_s)\epsilon(\bi) & \text{if $i_r=i_s=1, e=2$},\\
0 & \text{else}.
\end{cases}
\end{align*}

(9) Since $\xymatrix@C10pt{r\ar@{=}[r]|{\rangle}\ar@{=}[r]&s}$, by (\ref{depsiepsilon}), (\ref{depsitheta}) and (\ref{deftheta}), we deduce that
\begin{align*}
&\psi_r\psi_s\psi_r\psi_s\epsilon(\bi)\\
&=(\theta_r-\delta_{i_{r3}}^0y_r^{-1})(\theta_s-\delta_{i_{s2}}^0y_s^{-1})(\theta_r-\delta_{i_{r1}}^0y_r^{-1})(\theta_s-\delta_{i_{s}}^0y_s^{-1})\epsilon(\bi)\\
&=[(\theta_r\theta_s)^2-\delta_{i_s}^0y_{s}^{-1}\theta_r\theta_s\theta_r-\delta_{i_{r}}^0y_{r}^{-1}\theta_s\theta_r\theta_s+\delta_{i_{r}}^0\delta_{i_{s}}^0y_{r1}^{-1}y_{s}^{-1}\theta_r\theta_s+\delta_{i_{r}}^0\delta_{i_s}^0y_r^{-1}y_{s1}^{-1}\theta_s\theta_r\\
&-(\delta_{i_{r}}^0\delta_{i_s}^0y_s^{-1}y_{s1}^{-2}+\delta_{i_{r1}}^0y_{r1}^{-1}\theta_{s1}^{2})\theta_r-(2\delta_{i_{r}}^0\delta_{i_{s}}^0y_r^{-2}y_{r1}^{-1}+\delta_{i_{s1}}^0y_{s1}^{-1}\theta_{r}^{2})\theta_s+\delta_{i_r}^0\delta_{i_{s}}^0y_s^{-1}y_{s1}^{-1}\theta_{r}^2\\
&+\delta_{i_{r}}^0\delta_{i_{r1}}^0y_r^{-1}y_{r1}^{-1}\theta_s^2+\delta_{i_{r}}^0\delta_{i_{s}}^0y_r^{-2}y_s^{-2}]\epsilon(\bi).
\end{align*}	
The following is a similar calculation using (\ref{depsiepsilon}), (\ref{depsitheta}) and (\ref{deftheta}):
\begin{align*}&\psi_s\psi_r\psi_s\psi_r\epsilon(\bi)\\
&=[(\theta_s\theta_r)^2-\delta_{i_{s}}^0y_{s}^{-1}\theta_r\theta_s\theta_r-\delta_{i_r}^0y_{r}^{-1}\theta_s\theta_r\theta_s+\delta_{i_r}^0\delta_{i_{s}}^0y_{r1}^{-1}y_s^{-1}\theta_r\theta_s+\delta_{i_{r}}^0\delta_{i_{s}}^0y_r^{-1}y_{s1}^{-1}\theta_s\theta_r\\
&-(\delta_{i_r}^0\delta_{i_{s}}^0y_s^{-2}y_{s1}^{-1}+\delta_{i_{r1}}^0y_{r1}^{-1}\theta_{s}^{2})\theta_r-(2\delta_{i_r}^0\delta_{i_{s}}^0y_{r}^{-1}y_{r1}^{-2}+\delta_{i_{s1}}^0y_{s1}^{-1}\theta_{r1}^{2})\theta_s+\delta_{i_{r}}^0\delta_{i_{s}}^0y_s^{-1}y_{s1}^{-1}\theta_r^2\\
&+\delta_{i_r}^0\delta_{i_{r1}}^0y_r^{-1}y_{r1}^{-1}\theta_{s}^2+\delta_{i_r}^0\delta_{i_{s}}^0y_r^{-2}y_s^{-2}]\epsilon(\bi).
\end{align*}
By (\ref{detheta4}), (\ref{detheta^2}) and (\ref{depsitheta}), we know that
\begin{align*}
&\{[\psi_r\psi_s\psi_r\psi_s-\psi_s\psi_r\psi_s\psi_r]\epsilon(\bi)\\
&=[\delta_{i_{r}}^0\delta_{i_s}^0y_s^{-2}y_{s1}^{-2}y_r+\delta_{i_{r1}}^0y_{r1}^{-1}(\theta_s^{2}-\theta_{s1}^{2})]\theta_r-[\delta_{i_{r}}^0\delta_{i_{s}}^04y_r^{-2}y_{r1}^{-2}y_{s}+\delta_{i_{s1}}^0y_{s1}^{-1}(\theta_{r}^{2}-\theta_{r1}^2)]\theta_s\}\epsilon(\bi)\\
&=\begin{cases}-\psi_r\epsilon(\bi) & \text{if $i_s=1, i_r=-2, e\neq 2$};\\
\psi_r\epsilon(\bi)&\text{if $i_s=-1, i_r=2,e\neq 2$}; \\
(y_r\psi_r+1)\epsilon(\bi) & \text{if $i_r=0,i_s=1,e=2$},\\
2\psi_s\epsilon(\bi) & \text{if $i_r=-i_s=1,e\neq 2$};\\
-2\psi_s \epsilon(\bi)&\text{if $i_r=-i_s=-1, e\neq 2$}; \\
-4y_s\psi_s\epsilon(\bi) & \text{if $i_r=i_s=1,e=2$};\\
0 & \text{else}.
\end{cases}
\end{align*}

(10)  Since $\xymatrix@C10pt{r\ar@3{-}[r]|{\rangle}\ar@3{-}[r]&s}$, as a result of (\ref{depsiepsilon}), (\ref{depsitheta}) and (\ref{deftheta}), the following holds:	
\begin{align*}
&(\psi_r\psi_s)^3\epsilon(\bi)\\
&=(\theta_r-\delta_{i_{r5}}^0y_r^{-1})(\theta_s-\delta_{i_{s4}}^0y_s^{-1})(\theta_r-\delta_{i_{r3}}^0y_r^{-1})(\theta_s-\delta_{i_{s2}}^0y_s^{-1})(\theta_{r}-\delta_{i_{r1}}^0y_r^{-1})(\theta_s-\delta_{i_{s}}^0y_s^{-1})\epsilon(\bi)\\
&=\{(\theta_r\theta_s)^3-\delta_{i_s}^0y_{s}^{-1}(\theta_r\theta_s)^2\theta_r-\delta_{i_r}^0y_r^{-1}(\theta_s\theta_r)^2\theta_s+\delta_{i_r}^0\delta_{i_s}^0y_r^{-1}y_{s1}^{-1}(\theta_s\theta_r)^2+\delta_{i_{r}}^0\delta_{i_{s}}^0y_{r1}^{-1}y_s^{-1}(\theta_r\theta_s)^2\\
&-(\delta_{i_r}^0\delta_{i_{s}}^0y_s^{-1}y_{s2}^{-2}+\delta_{i_{r1}}^0y_{r1}^{-1}\theta_{s2}^{2})\theta_r\theta_s\theta_r-(\delta_{i_r}^0\delta_{i_{s}}^03y_r^{-2}y_{r2}^{-1}+\delta^0_{i_{s1}}y_{s1}^{-1}\theta_r^2)]\theta_s\theta_r\theta_s\\
&+\delta_{i_r}^0[\delta_{i_{s}}^0(3y_r^{-1}y_{r2}^{-2}y_{s1}^{-1}+y_r^{-2}y_{s1}^{-2}+2y_r^{-1}y_{r2}^{-1}y_s^{-2}-y_{s1}^{-1}y_{s2}^{-1}\theta_{r2}^2+y_{s}^{-1}y_{s1}^{-1}\sigma_r(\theta_{r}^{2}))\\
&+\delta_{i_{r2}}^0y_r^{-1}(y_{r1}^{-1}\theta_s^2-y_{r2}^{-1}\sigma_r\sigma_s(\theta_s^2))]\theta_r\theta_s+\delta_{i_r}^0[\delta_{i_s}^0(2y^{-2}_ry_{s2}^{-2}+y_{s1}^{-1}y_{s2}^{-1}\theta_r^2)+\delta_{i_{r1}}^0y_r^{-1}y_{r2}^{-1}\theta_{s2}^{2}]\theta_s\theta_r\\
&+[\delta_{i_{r}}^0\delta_{i_{s}}^0y_{r2}^{-1}(2y_{r}^{-1}y_{s1}^{-1}+y_{r2}^{-1}y_{s1}^{-1}-y_r^{-1}y_s^{-1})\theta_{s1}^2+\delta_{i_r}^0\delta_{i_{s}}^0y_r^{-1}y_{s1}^{-1}(y_s^{-1}y_{s1}^{-2}-y_{r}^{-1}y_{s}^{-2}-y_{r1}^{-1}\theta_s^2\\
&-y_ry_{s}^{-1}y_{s1}^{-1}\sigma_r(\theta_{r}^2))-\delta_{i_{s2}}^0y_{s2}^{-1}\theta_{r2}^{2}\theta_{s1}^2]\theta_r+[\delta_{i_r}^0\delta_{i_{s}}^0(y_r^{-1}y_s^{-1}y_{s2}^{-1}\theta_{r1}^2-3y_r^{-1}y_{r1}^{-1}y_{s1}^{-1}\theta_r^2-y_s^{-1}y_{s1}^{-2}\theta_r^2)\\	
&+\delta_{i_r}^0\delta_{i_{s}}^0y_r^{-2}y_{s}^{-1}(3y_{r1}^{-2}-y_r^{-1}y_s^{-1})-\delta_{i_{r2}}^0y_{r2}^{-1}\theta_{s1}^2\theta_r^2-\delta_{i_r}^0\delta_{i_{r1}}^0\delta_{i_{r2}}^0y_{r}^{-2}y_{r1}^{-1}\sigma_s(\theta_s^2))]\theta_s\\
&+\delta_{i_r}^0\delta_{i_{s}}^0(2y_{r}^{-2}y_{r1}^{-1}y_s^{-1}\theta_s^2-y_r^{-1}y_{r1}^{-2}y_{s}^{-1}\theta_s^2+y_{s}^{-2}y_{s1}^{-2}\theta_r^2+y_r^{-1}y_s^{-2}y_{s1}^{-1}\theta_{r}^2+y_r^{-3}y_s^{-3})\\
&+\delta_{i_{r2}}^0\delta_{i_{s}}^0y_{r2}^{-1}y_{s}^{-1}\theta_{s1}^2\theta_r^2+\delta_{i_{r1}}^0\delta_{i_{s1}}^0y_{r1}^{-1}y_{s1}^{-1}\theta_r^2\theta_{s}^{2}+\delta_{i_r}^0\delta_{i_{s2}}^0y_r^{-1}y_{s2}^{-1}\theta_{r1}^{2}\theta_s^2
\}\epsilon(\bi)
\end{align*}	
Imitate the proof above, we have that
\begin{align*}
&(\psi_s\psi_r)^3\epsilon(\bi)\\
&=(\theta_s-\delta_{i_{s5}}^0y_s^{-1})(\theta_r-\delta_{i_{r4}}^0y_r^{-1})(\theta_s-\delta_{i_{s3}}^0y_s^{-1})(\theta_r-\delta_{i_{r2}}^0y_r^{-1})(\theta_{s}-\delta_{i_{s1}}^0y_s^{-1})(\theta_r-\delta_{i_{r}}^0y_r^{-1})\epsilon(\bi)\\
&=\{[(\theta_s\theta_r)^3-\delta_{i_s}^0y_s^{-1}(\theta_r\theta_s)^2\theta_r-\delta_{i_r}^0y_{r}^{-1}(\theta_s\theta_r)^2\theta_s+\delta_{i_r}^0\delta_{i_s}^0
y_{r1}^{-1}y_s^{-1}(\theta_r\theta_s)^2+\delta_{i_{r}}^0\delta_{i_{s}}^0y_{r}^{-1}y_{s1}^{-1}(\theta_s\theta_r)^2\\
&-(\delta_{i_{r}}^0\delta_{i_s}^0y_s^{-2}y_{s2}^{-1}+\delta^0_{i_{r1}}y_{r1}^{-1}\theta_s^2)\theta_r\theta_s\theta_r-(\delta_{i_r}^0\delta_{i_{s}}^03y_r^{-1}y_{r2}^{-2}+\delta_{i_{s1}}^0y_{s1}^{-1}\theta_{r2}^{2})\theta_s\theta_r\theta_s+\delta_{i_r}^0[\delta_{i_{s}}^0(2y_{r2}^{-2}y_s^{-2}\\
&+y_s^{-1}y_{s2}^{-1}\theta_{r2}^{2})+\delta_{i_{r1}}^0y_{r1}^{-1}y_{r2}^{-1}\theta_s^2]
\theta_r\theta_s+\delta_{i_r}^0[\delta_{i_{s}}^0y_{r1}^{-1}(y_{r1}^{-1}y_{s2}^{-2}+y_s^{-2}y_{s2}^{-1})+\delta_{i_s}^0y_s^{-1}(2y_r^{-2}y_{s2}^{-1}+y_{s1}^{-1}\theta_r^2\\
&-y_{s2}^{-1}\sigma_s\sigma_r(\theta_r^2))+\delta_{i_{r1}}^0y_{r1}^{-1}(y_{r}^{-1}\sigma_s(\theta_{s}^{2})-y_{r2}^{-1}\theta_{s2}^2)]\theta_s\theta_r+[\delta_{i_{r}}^0\delta_{i_{s}}^0(y_r^{-1}y_{r2}^{-1}y_s^{-1}\theta_{s1}^2-2y_{r1}^{-2}y_{s1}^{-1}\theta_s^2\\
&-y_r^{-1}y_{r1}^{-1}y_s^{-1}\theta_s^2+y_r^{-1}y_s^{-2}y_{s1}^{-2}-y_r^{-2}y_s^{-3}-y_{s}^{-2}y_{s1}^{-1}\sigma_r(\theta_r^2))-\delta_{i_{s2}}^0y_{s2}^{-1}\theta_{r1}^2\theta_s^2]\theta_r+[\delta_{i_{r}}^0\delta_{i_{s}}^0(2y_{r1}^{-1}y_{s}^{-1}y_{s2}^{-1}\\
&+y_{r1}^{-1}y_{s2}^{-2}-y_r^{-1}y_s^{-1}y_{s2}^{-1})\theta_{r1}^2-\delta_{i_{r2}}^0y_{r2}^{-1}\theta_{s2}^{2}\theta_{r1}^2-\delta_{i_r}^0\delta_{i_{r1}}^0\delta_{i_{r2}}^0y_{r1}^{-2}y_{r}^{-1}\sigma_s(\theta_{s}^2)+\delta_{i_{r}}^0\delta_{i_s}^0(3y_r^{-1}y_{r1}^{-3}y_s^{-1}\\
&-y_{r1}^{-1}y_s^{-1}y_{s1}^{-1}\theta_r^2-y_{r}^{-2}y_{r1}^{-1}y_{s}^{-2})]\theta_s+\delta_{i_r}^0\delta_{i_s}^0[(2y_r^{-1}y_{s}^{-2}y_{s1}^{-1}-y_{r}^{-1}y_s^{-1}y_{s1}^{-2})\theta_r^2+(3y_{r}^{-2}y_{r1}^{-2}\\
&+y_r^{-2}y_{r1}^{-1}y_s^{-1})\theta_{s}^2+y_r^{-3}y_s^{-3}]+(\delta_{i_{r}}^0\delta_{i_{s2}}^0y_{r}^{-1}y_{s2}^{-1}\theta_{r1}^2+\delta_{i_{r1}}^0\delta_{i_{s1}}^0y_{r1}^{-1}y_{s1}^{-1}\theta_{r}
^{2})\theta_s^2+\delta_{i_{r2}}^0\delta_{i_s}^0y_{r2}^{-1}y_s^{-1}\theta_{s1}^{2}\theta_r^2
\}\epsilon(\bi)
\end{align*}	
Using relations (\ref{detheta6}), (\ref{detheta^2}) and  (\ref{depsitheta}), we see that
\begin{align*}
&[(\psi_r\psi_s)^3-(\psi_s\psi_r)^3]\epsilon(\bi)\\
&=\{[\delta_{i_{r}}^0\delta_{i_s}^0y_s^{-2}y_{s2}^{-2}y_{s1}+\delta_{i_{r1}}^0y_{r1}^{-1}(\theta_s^2-\theta_{s2}^2)]\theta_r\theta_s\theta_r-[\delta_{i_{r}}^0\delta_{i_s}^03y_r^{-2}y_{r2}^{-2}y_{r1}+\delta_{i_{s1}}^0y_{s1}^{-1}(\theta_{r}^2-\theta_{r2}^2)]\theta_s\theta_r\theta_s\\
&+[\delta_{i_r}^0\delta_{i_s}^0(3y_{r}^{-1}y_{r2}^{-2}y_{s1}^{-1}+y_r^{-2}y_{s1}^{-2}+2y_{r}^{-1}y_{r2}^{-2}y_s^{-2}y_{r1})+\delta_{i_r}^0\delta_{i_{s}}^0y_s^{-1}y_{s1}^{-1}(\sigma_r(\theta_r^2)-\theta_{r2}^2)\\
&+\delta_{i_r}^0\delta_{i_{r1}}^0y_r^{-1}y_{r2}^{-1}(\theta_s^2-\sigma_r\sigma_s(\theta_s^2))]\theta_r\theta_s-[\delta_{i_r}^0\delta_{i_s}^0(2y_r^{-2}y_{s}^{-1}y_{s2}^{-2}y_{s1}+y_{r1}^{-2}y_{s2}^{-2}+y_{r1}^{-1}y_{s}^{-2}y_{s2}^{-1})\\
&+\delta_{i_r}^0\delta_{i_{r1}}^0y_r^{-1}y_{r1}^{-1}(\sigma_s(\theta_s^2)-\theta_{s2}^2)+\delta_{i_r}^0\delta_{i_{s}}^0y_s^{-1}y_{s2}^{-1}(\theta_r^2-\sigma_s\sigma_r(\theta_r^2))]\theta_s\theta_r+[\delta_{i_r}^0\delta_{i_{s}}^0(y_{r1}^{-1}y_s^{-1}y_{s1}^{-1}\theta_s^2\\
&+2y_{r1}^{-2}y_{s1}^{-1}\theta_s^2+y_{r2}^{-2}y_{s1}^{-1}\theta_{s1}^{2}-2y_{r2}^{-1}y_{s}^{-1}y_{s1}^{-1}\theta_{s1}^{2}+y_{s}^{-2}y_{s1}^{-2}y_r\sigma_r(\theta_{r}^2)+y_s^{-3}y_{s1}^{-2}+y_r^{-1}y_s^{-1}y_{s1}^{-3})\\
&+\delta_{i_{s2}}^0y_{s2}^{-1}(\theta_{r1}^2\theta_s^2-\theta_{r2}^{2}\theta_{s1}^2)]\theta_r+[\delta_{i_r}^0\delta_{i_{s}}^0(6y_{r}^{-1}y_{r1}^{-1}y_{s2}^{-1}\theta_{r1}^{2}-y_{r1}^{-1}y_{s2}^{-2}\theta_{r1}^{2}+9y_r^{-2}y_{r1}^{-3}-2y_{r1}^{-1}y_{s1}^{-2}\theta_r^2\\
&-3y_r^{-1}y_{r1}^{-1}y_{s1}^{-1}\theta_r^2-3y_r^{-3}y_{r1}^{-1}y_{s}^{-1})-\delta_{i_r}^0\delta_{i_{r1}}^03y_{r}^{-2}y_{r1}^{-2}y_s\sigma_s(\theta_{s}^2)+\delta_{i_{r2}}^0y_{r2}^{-1}(\theta_{r1}^2\theta_{s2}^2-\theta_{r}^{2}\theta_{s1}^2)]\theta_s\}\epsilon(\bi)\\	
&=\begin{cases}  \delta_{i_{r1}}^0\frac{\theta_s^2-\theta_{s2}^2}{y_{r1}}\theta_r\theta_s\theta_r\epsilon(\bi) & \text{if $i_{r1}=0,i_r\neq 0,i_{s1}\neq 0$},\\
\delta_{i_{s1}}^0\frac{\theta_{r2}^2-\theta_{r}^2}{y_{s1}}\theta_s\theta_r\theta_s\epsilon(\bi) & \text{if $i_{s1}=0,i_r\neq 0,i_{r1}\neq 0$},\\
\delta_{i_{s2}}^0\frac{\theta_{r1}^2\theta_s^2-\theta_{r2}^{2}\theta_{s1}^2}{y_{s2}}\theta_r\epsilon(\bi) & \text{if $i_{s2}=0,i_r\neq 0$},\\
\delta_{i_{r2}}^0\frac{\theta_{r1}^2\theta_{s2}^2-\theta_{r}^{2}\theta_{s1}^2}{y_{r2}}\theta_s\epsilon(\bi) & \text{if $i_{r2}=0,i_r\neq 0$},\\
\delta_{i_{s2}}^0\frac{\theta_{r1}^2\theta_s^2-\theta_{r2}^{2}\theta_{s1}^2}{y_{s2}}\theta_r\epsilon(\bi) & \text{if $i_{s2}=0,i_r=0,i_s\neq 0$},\\
\delta_{i_{r2}}^0\frac{\theta_{r1}^2\theta_{s2}^2-\theta_{r}^{2}\theta_{s1}^2}{y_{r2}}\theta_s\epsilon(\bi) & \text{if $i_{r2}=0,i_s=0,i_r\neq 0$},\\
(\delta_{i_{r1}}^0\frac{\theta_s^2-\theta_{s2}^2}{y_{r1}}\theta_r\theta_s\theta_r+\delta_{i_{s1}}^0\frac{\theta_{r2}^2-\theta_{r}^2}{y_{s1}}\theta_s\theta_r\theta_s)\epsilon(\bi) & \text{if $i_{r1}=0,i_{s1}=0,i_{r}\neq 0$},\\
(\delta_{i_{r1}}^0\frac{\theta_s^2-\theta_{s2}^2}{y_{r1}}\theta_r\theta_s\theta_r+\delta_{i_r}^0\delta_{i_{r1}}^0\frac{\sigma_r\sigma_s(\theta_s^2)-\theta_s^2}{\sigma_r(y_r)y_{r2}}\theta_r\theta_s\\
+\delta_{i_r}^0\delta_{i_{r1}}^0\frac{\theta_{s2}^2-\sigma_s(\theta_s^2)}{y_ry_{r1}}\theta_s\theta_r-\delta_{i_r}^0\delta_{i_{r1}}^0\frac{3y_s\sigma_s(\theta_{s}^2)}{y_{r}^{2}y_{r1}^{2}}\theta_s\\
+\delta_{i_{r2}}^0\frac{\theta_{r1}^2\theta_{s2}^2-\theta_{r}^{2}\theta_{s1}^2}{y_{r2}}\theta_s)\epsilon(\bi) & \text{if $i_r=0,i_{r1}=0,i_{r2}=0,i_{s}\neq 0$},\\
0&\text{else}
\end{cases}	\\
&=\begin{cases}-\psi_r\psi_s\psi_r\epsilon(\bi) & \text{if $i_s=1,i_{r}=-3,e\neq 2,3$},\\
\psi_r\psi_s\psi_r\epsilon(\bi) & \text{if $i_s=-1,i_r=3, e\neq 2,3$},\\
3\psi_s\psi_r\psi_s\epsilon(\bi) & \text{if $i_r=-i_s=1,e\neq 2$},\\
-3\psi_s\psi_r\psi_s\epsilon(\bi) & \text{if $i_r=-i_s=-1,e\neq 2$},\\
-2y_r\psi_r\epsilon(\bi) & \text{if $i_s=\pm 1,i_{r}=\mp 2,e\neq 2$},\\
-2y_s\psi_s\epsilon(\bi) & \text{if $i_r=\pm 1,i_{s}= 2, e=4$},\\
\psi_s\epsilon(\bi) & \text{if $i_r=1,i_{s}\neq -2, e\neq 2$},\\
-\psi_s\epsilon(\bi) & \text{if $i_r=-1,i_{s}\neq -2, e\neq 2$},\\
\psi_s\epsilon(\bi) & \text{if $i_r=3,i_{s}=-2, e\neq 2,3$},\\
-\psi_s\epsilon(\bi) & \text{if $i_r=-3,i_{s}=2, e\neq 2,3$},\\
-4(y_{r}^2+3y_sy_{s1})(y_r\psi_r+1)\epsilon(\bi) & \text{if $i_r=0,i_s=1, e=2$},\\
4(y_{r}^2+3y_sy_{s1})(y_s\psi_s+1)\epsilon(\bi) & \text{if $i_r=1, i_s=0, e=2$},\\
(y_{s1}\psi_r\psi_s\psi_r-3y_{r1}\psi_s\psi_r\psi_s-y_r^2+3y_s^2)\epsilon(\bi) & \text{if $i_r=i_s=1, e=2$},\\
-\psi_r\psi_s\psi_r\epsilon(\bi) & \text{if $i_r=0,i_s=1, e=3$},\\
\psi_r\psi_s\psi_r\epsilon(\bi) & \text{if $i_r=0,i_s=-1, e=3$},\\
0 & \text{else}.
\end{cases}
\end{align*}
	
To finish the proof of the theorem, we need to prove the relations (1)-(10) generate all relations. In fact, for each $w\in \mcW$, we {\it fix} a reduced decomposition $w=\sigma_{r_1} \sigma_{r_2}\cdots \sigma_{r_m}$ and define the element
$$\psi_w:=\psi_{r_1}\psi_{r_2}\cdots \psi_{r_m}\in \mathcal{L}.$$
Note that $\psi_w$ in general does depend on the choice of reduced decomposition of $w$ \cite[Proposition 2.5]{BKW}. Using the decomposition (\ref{de Lusztig theta_w}), $\mathcal{L}$ has a $\Bbbk(y)\otimes_\Bbbk\mathcal{E}$-basis $\{\theta_w\ | \ w\in \mcW\}$. By (\ref{depsitheta}), we know that $\{\psi_w \ | \ w\in \mcW\}$ is also a basis of $\mathcal{L}$ as $\Bbbk(y)\otimes_\Bbbk\mathcal{E}$-module. Thus the relations (1)-(9) is complete since by them every element in $\mathcal{L}$ can be written as $\sum_{w\in \mcW,\bi\in \mcC}\psi_wf_{w,\bi}(y)g^{-1}_{w,\bi}(y)\epsilon(\bi)$ with $f_{w,\bi}(y),g_{w,\bi}(y)\in \Bbbk[y]$ and $g_{w,\bi}(y)\neq 0$.
\end{proof}

\subsection{The BK subalgebras}
Our method of constructing the KLR generators of the Lusztig extension follows Brundan and Kleshchev. With these KLR generators in hand, we can fast obtain a class of BK-type isomorhism. In the process, the key point is the following subalgebra of the Lusztig extension using its KLR form. We think it is suitable to call it Brundan Kleshchev subalgebra (or BK subalgebra for short).

We define the {\it BK subalgebra $\tilde{\mathcal{L}}$} as the $\Bbbk$-algebra generated by $$\{y_1,\cdots,y_n, \psi_1,\cdots,\psi_n, f^{-1}(y),\epsilon(\bi) \ |\ \bi\in \mcC, f(y)\in \Bbbk[y] \ \mbox{with} \ f(0)\neq 0\}$$
subject to the relations (1)-(10) of Theorem \ref{thm:deKLR basis}. It is a subalgebra of $\mathcal{L}$ as shown in the following corollary.
\begin{corollary}\label{cor:basis deBK algebra}  Denote by $1\epsilon(\bi)=\epsilon(\bi)$. Then the algebra $\tilde{\mathcal{L}}$ is generated by $$\{x_1,\cdots,x_n, t_1,\cdots,t_{n}, f^{-1}(x)\epsilon(\bi) \ |\ \bi\in \mcC, f(x)\in \Bbbk[x] \ \mbox{with} \ f(\bi)\neq 0\}$$
	subject to relations (\ref{deaffhecke x})-(\ref{detttttt}), (\ref{deepsilon}), (\ref{dexepis}), (\ref{tepsi}) and for $f\in \Bbbk[X]$ with $f(\bi)\neq 0$
	\begin{equation}\label{def(x)f^{-1}(x)}
	\epsilon(\bj)\cdot f^{-1}\epsilon(\bi)=\delta_{\bi}^{\bj}f^{-1}\epsilon(\bi)=f^{-1}\epsilon(\bi)\cdot \epsilon(\bj), \ f\cdot f^{-1}\epsilon(\bi)=\epsilon(\bi)=f^{-1}\epsilon(\bi)\cdot f.
	\end{equation}
\end{corollary}
\begin{proof} There is an obvious homomorphism $\alpha\colon \tilde{\mathcal{L}}\to \mathcal{L}$ by sending generators to the same named generators. This homomorphism is injective since using relations (1)-(9) of Theorem \ref{thm:deKLR basis}, every element in $\tilde{\mathcal{L}}$ can be written as
	$$\sum_{w\in \mcW,\bi\in \mcC}\psi_wf_{w,\bi}(y)\cdot g_{w,\bi}^{-1}(y)\epsilon(\bi)$$ with $f_{w,\bi}(y), g_{w,\bi}(y)$ in $\Bbbk[y]$ and $g_{w,\bi}(0)\neq 0$, and $\{\psi_w\ | \ w\in \mcW\}$ is a $\Bbbk(y)\otimes_\Bbbk\mathcal{E}$-basis of $\mathcal{L}$. Thus $\mathrm{Im}\alpha=\oplus_{w\in \mcW}\psi_w\mathcal{P}(y,\mathcal{E})$, where $\mathcal{P}(y,\mathcal{E})$ is the commutative algebra $\{fg^{-1}\ | \ f,g\in\Bbbk[y], g(0)\neq0\}\otimes_\Bbbk\mathcal{E}$.
		
	Assume $\mathcal{M}$ is the $\Bbbk$-algebra generated by 	$$\{x_1,\cdots,x_n, t_1,\cdots,t_{n}, f^{-1}(x)\epsilon(\bi) \ |\ \bi\in \mcC, f(x)\in \Bbbk[x] \ \mbox{with} \ f(\bi)\neq 0\}$$
	subject to relations (\ref{deaffhecke x})-(\ref{detttttt}), (\ref{deepsilon}), (\ref{dexepis}), (\ref{tepsi}) and (\ref{def(x)f^{-1}(x)}). Then, using these relations, every element in $\mathcal{M}$ can be written as $$\sum_{w\in \mcW,\bi\in \mcC}T_wf_{w,\bi(x)}\cdot g_{w,\bi}^{-1}(x)\epsilon(\bi)$$ with $f_{w,\bi}(x), g_{w,\bi}(x)$ in $\Bbbk[x]$ and $g_{w,\bi}(\bi)\neq 0$. Thus there also has an injective homomorphism $\alpha'\colon \mathcal{M}\to \mathcal{L}$ by sending generators to the same named generators. So $\mathrm{Im}\alpha'=\oplus_{w\in \mcW}T_w\mathcal{P}(x,\mathcal{E})$, where $\mathcal{P}(x,\mathcal{E})$ is the commutative algebra $\{f\cdot g^{-1}\epsilon(\bi) \ | \  \bi\in \mcC,f,g\in \Bbbk[x], g(\bi)\neq 0\}$.
	
	We claim that $\mathrm{Im}\alpha=\mathrm{Im}\alpha'$. In fact, by Lemma \ref{lem:K(x)=K(y)}, $\mathcal{P}(x,\mathcal{E})=\mathcal{P}(y,\mathcal{E})$ in $\mathcal{L}$. Notice that
	\begin{align*}\psi_r=\sum\limits_{\begin{subarray}{1}\bi\in\mcC \\
		i_r\neq 0\end{subarray}}(t_r+\tfrac{1}{y_r+i_r})q_r^{-1}(\bi)\epsilon(\bi)+\sum\limits_{\begin{subarray}{1}\bi\in\mcC \\
		i_r=0\end{subarray}}(t_r+1)q^{-1}_r(\bi)\epsilon(\bi).\end{align*}
	and $q_r(\bi), q_r^{-1}(\bi)\in \mathcal{P}(y,\mathcal{E})$ in $\mathcal{L}$, thus we get $\mathrm{Im}\alpha=\mathrm{Im}\alpha'$ and then  $\mathcal{M}\cong \tilde{\mathcal{L}}$.
\end{proof}

Denote by $$\tilde{\mathcal{L}}(\Lambda):=\tilde{\mathcal{L}}/\langle y_1^{\Lambda_{i_1}}\epsilon(\bi) | \bi\in \mcC \rangle.$$
We use the same letters $\psi_1,\cdots,\psi_n$ and $y_1,\cdots,y_n$ to denote the images of the generators in $\tilde{\mathcal{L}}(\Lambda)$.

\begin{lemma} \label{lem:cdesemiration} 	$\tilde{\mathcal{L}}(\Lambda)=\tilde{\mathcal{L}}/\langle \prod_{i\in I}(x_1-i)^{\Lambda_i}\rangle.$
\end{lemma}
\begin{proof} By the relations of $y_1$ and $x_1$, there holds that \begin{align*}\prod_{i\in I}(x_1-i)^{\Lambda_i}&=\sum_{\bj\in \mcC}\prod_{i\in I}(y_1+j_1-i)^{\Lambda_i}\epsilon(\bj)\\
	&=\sum_{\bj\in \mcC}\prod_{i\in I,i\neq j_1}(y_1+j_1-i)^{\Lambda_i}y_1^{\Lambda_{j_1}}\epsilon(\bj)
	\end{align*} is in $\langle y_1^{\Lambda_{j_1}}\epsilon(\bj)\ | \ \bj\in \mcC\rangle$. By Corollary (\ref{cor:basis deBK algebra}), $\prod_{i\in I, i\neq j_1}[(x_1-i)^{\Lambda_i}]^{-1}\epsilon(\bj)$ is in $\tilde{\mathcal{L}}$, thus \begin{align*}y_1^{\Lambda_{j_1}}\epsilon(\bj)&=(x_1-j_1)^{\Lambda_{j_1}}\epsilon(\bj)\\
	&=\prod_{i\in I}(x_1-i)^{\Lambda_i}\prod_{i\in I, i\neq j_1}[(x_1-i)^{\Lambda_i}]^{-1}\epsilon(\bj)
	\end{align*}
	is in $\langle \prod_{i\in I}(x_1-i)^{\Lambda_i}\rangle$. Therefore $\langle y_1^{\Lambda_{i_1}}\epsilon(\bi)\ | \ \bi\in \mcC\rangle =\langle \prod_{i\in I}(x_1-i)^{\Lambda_i}\rangle$ in $\tilde{\mathcal{L}}.$ Thus $\tilde{\mathcal{L}}(\Lambda)=\tilde{\mathcal{L}}/\langle \prod_{i\in I}(x_1-i)^{\Lambda_i}\rangle.$
\end{proof}
\begin{remark}
	Similarly to the proof of \cite[Lemma 2.1]{Brundan-Kleshchev09}, one can deduce that the elements $y_r$ are nilpotent in $\tilde{\mathcal{L}}(\Lambda)$. Then by imitating the proof above, we know that the elements $\prod_{i\in I}(x_r-i)$ are also nilpotent in $\tilde{\mathcal{L}}(\Lambda)$.
\end{remark}

\subsection{Cyclotomic degenerate affine Hecke algebras}

We define the {\itshape cyclotomic degenerate  affine Hecke algebra $\mcH(\Lambda)$} as
$$\mcH(\Lambda):=\mcH/\langle\prod_{i\in I}(x_1-i)^{\Lambda_i}\rangle.$$
Similar to \cite[Subsection 3.1]{Brundan-Kleshchev09}, there is a system $\{e(\bi) \ | \ \bi\in \mcC\}$ of mutually orthogonal idempotents in $\mcH(\Lambda)$ such that $1=\sum_{\bi\in I^n}e(\bi)$ and \begin{align*}e(\bi)\mcH(\Lambda)&=\{h\in \mcH(\Lambda) \ | \ (x_r-i_r)^mh=0 \ \mbox{for all} \ r\in [n] \ \mbox{and} \ m\gg0\}.\end{align*}
It is easy to check that $x_re(\bi)=e(\bi)x_r$ for all $r\in [n]$ and $\bi\in \mcC$, and for a polynomial $f(x)\in \Bbbk[x]$, $f(x)e(\bi)$ is a unit if and only if $f(\bi)\neq 0$. In particular, the element $x_re(\bi)$ is a unit in $e(\bi)\mcH(\Lambda)$ if and only if $i_r\neq 0$. In this case, we write $x_r^{-1}e(\bi)$ for the inverse.

\begin{lemma} \label{lem:degenerate} The following relations hold for all $r\in [n]$ and $\bi\in I^{n}$.
	\begin{equation} \label{detre-etr} t_re(\bi)=\begin{cases}
	e(\bi)t_r & \text{if $i_r=0$},\\
	e(\sigma_r(\bi))t_r+
	x_r^{-1}e(\sigma_r(\bi))-x_r^{-1}e(\bi)  & \text{if $i_r\neq 0$}.
	\end{cases}
	\end{equation}
\end{lemma}

\begin{proof}   For any $s\in [n]$, the element $$(\sigma_r(x_s)-\sigma(\bi)_s)e(\bi)=[(x_s-i_s)-a_{sr}(x_r-i_r)]e(\bi)$$ is nilpotent by the nilpotency of $(x_s-i_s)e(\bi)$ and $(x_r-i_r)e(\bi)$. Similarly, we can show that
		$\parti_r((x_s-i_s)^m)e(\bi)$ is nilpotent by the binomial theorem for an integer $m\gg0$ whenever $i_r=0$. Therefore, if $i_r=0$, by (\ref{deft}), we have
		\begin{align*}(x_s-i_s)^mt_re(\bi)=t_r(\sigma_r(x_s)-i_s)^me(\bi)+{\parti}_r((x_s-i_s)^m)e(\bi)=0\end{align*}
		when $m\gg0$. Hence $t_re(\bi)\in e(\bi)\mcH(\Lambda)$ and then $t_re(\bi)=e(\bi)t_re(\bi)=e(\bi)t_r.$
		
		If $i_{r}\neq 0$, by (\ref{deft}), we get that
		\begin{align*}(x_s-\sigma_r(\bi)_s)^m[t_rx_r+1]e(\bi)=[t_rx_r+1](\sigma_r(x_s)-\sigma_r(\bi)_s)^me(\bi)=0\end{align*}
		when $m\gg0$. Therefore
		\begin{align*}t_rx_re(\bi)+e(\bi)=e(\sigma_r(\bi))[t_rx_re(\bi)+e(\bi)]=e(\sigma(\bi))t_rx_{r}e(\bi).
		\end{align*}
		Then right-multiplying by $x_r^{-1}e(\bi)$, we obtain that $$t_re(\bi)=e(\sigma_r(\bi))t_re(\bi)-
		x_r^{-1}e(\bi).$$
		Similarly, we deduce that
		$$e(\sigma_r(\bi))t_r=e(\sigma_r(\bi))t_re(\bi)-
		x_r^{-1}e(\sigma_r(\bi)).$$
		Therefore $t_re(\bi)=e(\sigma_r(\bi))t_r+
		x_r^{-1}e(\sigma_r(\bi))-x_r^{-1}e(\bi).$
	\end{proof}

Let $$e(\mcC):=\sum_{\bi\in \mcC}e(\bi)\in \mcH(\Lambda).$$
Then $e(\mcC)$ is a central idempotent in $\mcH(\Lambda)$ by (\ref{detre-etr}). Furthermore, by Lemma \ref{lem:cdesemiration}, Corollary \ref{cor:basis deBK algebra} and Lemma \ref{lem:degenerate}, there is a homomorphism
$$\rho \colon \tilde{\mathcal{L}}(\Lambda) \to \mcH(\Lambda)e(\mcC)$$ sending the generators $x_r, t_r$ to the same named elements, and $f(x)^{-1}\epsilon(\bi)$ with $f(\bi)\neq 0$ to $f^{-1}(x)e(\bi)$. Now we arrive at our first main result in this article for degenerate affine Hecke algebras.
\begin{theorem}\label{thm:derho}There is an algebra isomorphism $\mcH(\Lambda)e(\mcC)\cong \tilde{\mathcal{L}}(\Lambda)$.
\end{theorem}

\begin{proof}  Apparently, $\rho$ is surjective, thus we only need to construct a left-inverse of $\rho$. By lemma \ref{lem:cdesemiration}, there is a homomorphism
	$$\tau \colon \mcH(\Lambda) \to \tilde{\mathcal{L}}(\Lambda)$$
	sending the generators $x_r,t_r$ to the same named elements. Let $\bi\in \mcC $ and $\bj\in I^n$. If $\bi\neq \bj$, then there is some $1\leq r\leq n$ such that $j_r\neq i_r$. We claim that $\epsilon(\bi)\tau(e(\bj))=0$. In fact, by the construction of $e(\bj)$, there is an integer $m\gg0$ such that $(x_r-j_r)^me(\bj)=0$. So we have $$(x_r-j_r)^m\epsilon(\bi)\tau(e(\bj))=\epsilon(\bi)\tau((x_r-j_r)^me(\bj))=0.$$
	The assumption $j_r\neq i_r$ implies that the element $(x_r-j_r)^{-1}\epsilon(\bi)\in \tilde{\mathcal{L}}(\Lambda)$. Thus we deduce that
	$$\epsilon(\bi)\tau(e(\bj))=(x_r-j_r)^{-m}(x_r-j_r)^m\epsilon(\bi)\tau(e(\bj))=(x_r-j_r)^{-m}\epsilon(\bi)0=0.$$
	Therefore, if $\bj\in I^n\setminus \mcC$ we have that $\tau(e(\bj))=\sum_{\bi\in \mcC}\epsilon(\bi)\tau(e(\bj))=0.$ Moreover, if $\bj\in \mcC$, we get that
	$$\tau(e(\bj))=\sum_{\bi\in \mcC}\epsilon(\bi)\tau(e(\bj))=\epsilon(\bj)\tau(e(\bj))=\epsilon(\bj)\sum_{\bi\in I^n}\tau(e(\bi))=\epsilon(\bj)\tau(1)=\epsilon(\bj).$$
	These show that $\tau|_{ \mcH(\Lambda)e(\mcC)}\colon \mcH(\Lambda)e(\mcC)\to \tilde{\mathcal{L}}(\Lambda)$ is an algebra homomorphism. It is easy to check that $\tau \rho$ is the identity on each generator of $\tilde{\mathcal{L}}(\Lambda)$. Thus $\rho$ is an isomorphism.
\end{proof}

\subsection{The cyclotomic BK subalgbra revisited} Following \cite[Subsection 2.2]{Brundan-Kleshchev09}, we introduce an algebra $\mcR$ which is defined to be the $\Bbbk$-algebra $\mcR$ generated by $$\{y_1,\cdots,y_n, \psi_1,\cdots,\psi_n, \epsilon(\bi) \ |\ \bi\in \mcC,\}$$
subject to the relations (1)-(2) and (4)-(10) of Theorem \ref{thm:deKLR basis}. Similar to the Lusztig extension $\mathcal{L}$, by the defining relations, $\mcR$ has a basis $\{\psi_w\ | \ w\in \mcW\}$ as a $\Bbbk[y]\otimes_\Bbbk \mathcal{E}$-module. Moreover, it can be viewed as a subalgebra of $\mathcal{L}$.

We denote by $\mcR(\Lambda):=\mcR/\langle y_1^{\Lambda_{i_1}}\epsilon(\bi) | \bi\in \mcC \rangle$. Then we have the following isomorphism of cyclotomic algebras.
\begin{proposition} \label{prop:cyclotomic semi-KLRA} We have $\Bbbk$-algebra isomorphsim $\mcR(\Lambda)\cong \tilde{\mathcal{L}}(\Lambda)$.
\end{proposition}
\begin{proof}  Similarly to the proof of \cite[Lemma 2.1]{Brundan-Kleshchev09}, we can also show that the elements $y_r$ are nilpotent in $\mcR(\Lambda)$. Thus if $f(y)\in \Bbbk[y]$ with $f(0)\neq0$, the polynomial $f(y)-f(0)$ is nilpotent in $\tilde{\mathcal{L}}(\Lambda)$ by the nilpotency of the elements $y_r$. So there exists some $g(y)\in \Bbbk[y]$ and $m\in \mathbb{N}$ such that $g(y)^m=0$ and $f^{-1}(y)=f(0)^{-1}\sum_{l=0}^m g(y)^l$ in $\mcR(\Lambda)$. Therefore, the  homomorphism $\mcR\hookrightarrow \tilde{\mathcal{L}}\twoheadrightarrow \tilde{\mathcal{L}}(\Lambda)$ is surjective and induces a surjective homomorphism $\pi\colon \mcR(\Lambda)\to \tilde{\mathcal{L}}(\Lambda)$. Let $\tilde{F}$ be the localization of the commutative ring $\Bbbk[y]^{\mcW}$ of $\mcW$-invariants in $\Bbbk[y]$ with respect to $\{f\in \Bbbk[y]^{\mcW} \  | \ f(0)\neq 0\}$. Similar to the proof of (\ref{deK(x)}), we can show that $\tilde{\mathcal{L}}\cong \mathcal{R}\otimes_{\Bbbk[y]^{\mcW}}\tilde{F}$. Since the elements $y_r$ are nilpotent in
	$\mathcal{R}(\Lambda)$, similar to the proof above, we know that if $f(y)\in K[y]^{\mcW}$ with $f(0)\neq 0$, then it is a unit in $\mathcal{R}(\Lambda)$. Thus the homomorphism $\Bbbk[y]^{\mcW}\hookrightarrow \mathcal{R}\stackrel{\pi_1}\twoheadrightarrow \mathcal{R}(\Lambda)$ induces a morphism $\pi_2\colon \tilde{F}\to \mathcal{R}(\Lambda)$. Therefore we have an induced algebra homomorphism $\pi_1\otimes \pi_2\colon \tilde{\mathcal{L}}\to \mathcal{R}(\Lambda)$. The homomorphism $\pi_1\otimes\pi_2$ induces an algebra homomorphism $\pi'\colon \tilde{\mathcal{L}}(\Lambda)\to \mathcal{R}(\Lambda)$. It is easy to check that $\pi$ and $\pi'$ are two-sided inverses. So $\mcR(\Lambda)\cong \tilde{\mathcal{L}}(\Lambda)$.
\end{proof}

By Theorem \ref{thm:derho} and the Proposition above, we get a Brundan-Kleshchev type isomorphism for the degenerate affine Hecke algebra.
\begin{corollary}\label{cor:cyclotomic} We have $\Bbbk$-algebra isomorphism $\mcH(\Lambda)e(\mcC)\cong \mcR(\Lambda)$.
\end{corollary}

\section{The BK subalgebra in non-degenerate case} In this section, we assume that $e$ is the smallest positive integer such that $1+q+\cdots +q^{e-1}=0$ and setting $e=0$ if no such that integer exists.

\subsection{The non-degenerate affine Hecke algebras} We define the {\itshape non-degenerate affine Hecke algebra} $\mcH_q$ to be the unital $\Bbbk$-algebra with generators $\{X_r^{\pm 1}, T_r  |  r\in [n]\}$ subject to the following relations for all admissible indices:
\begin{align}
\label{affhecke X}
X_rX_s&=X_sX_r,\quad X_rX_r^{-1}=X_r^{-1}X_r=1;
\\
\label{TX}
T_r X_s &= \sigma_r(X_s)T_r+(1-q)\mathrm{D}_r(X_s);\\
\label{relation:affhecke Tq}
\quad (T_r + 1)(T_r - q) &= 0;
\\
\label{relation:affinehecke_T}
T_r  T_s &= T_s  T_r,\hspace{43.4mm} \text{if $r \nslash s$};
\\
\label{TTT}
T_r T_sT_r &= T_sT_rT_s, \hspace{39.4mm} \text{if $r -\- s$};
\\
\label{TTTT}
(T_r T_s)^2 &= (T_sT_r)^2, \hspace{39mm} \text{if $\xymatrix@C10pt{r\ar@{=}[r]|{\rangle}\ar@{=}[r]&s}$};	
\\
\label{TTTTTT}
(T_r T_s)^3 &= (T_sT_r)^3, \hspace{39.4mm} \text{if $\xymatrix@C10pt{r\ar@3{-}[r]|{\rangle}\ar@3{-}[r]&s}$}.	
\end{align}

If $w=\sigma_{r_1} \sigma_{r_2}\cdots \sigma_{r_m}$ is a reduced expression in $\mcW$, then $$T_w:=t_{r_1}t_{r_2}\cdots t_{r_m}$$ is a well-defined element in $\mcH_q$. By \cite[Lemma 3.4]{Lusztig89}, the algebra $\mcH_q$ has Bernstein-Zelevinski basis $\{T_w \ | \ w\in \mcW\}$ as $\Bbbk[X^{\pm1}]$-module.
	
\subsection{The rationalization of $\mcH_q$} Recall \cite[proposition 3.11]{Lusztig89} that the center of $\mcH_q$ is $\mathcal{Z}$. Then $\mcH$ can be seen as a $\mcZ$-subalgebra (identified with the subspace $\mcH_q\otimes 1$) of the $\mcZ$-algebra $$\mcH_{q,\mcF}:=\mcH_q\otimes_{\mcZ} \mcF.$$ Moreover, we can identify the rational polynomial field $\Bbbk(X)$ as a subspace of $\mcH_{q,\mcF}$ via the natural isomorphism (\ref{K(X)}). For any $r\in [n]$ and $f\in \Bbbk(X)$, as a consequence of \cite[Subsection 3.12 (d)]{Lusztig89}, there holds that
\begin{equation} \label{fT} T_rf=\sigma_r(f)T_r+(1-q)\mathrm{D}_r(f).\end{equation}

\subsection{Intertwining elements}  For $r\in [n]$, we define the {\itshape intertwining element} $\Phi_r$ in $\mcH_{q,\mcF}$ as follows:
$$\Phi_r:=T_r+(1-q)(1-X_r)^{-1}.$$
The following result is the non-degenerate version of Proposition \ref{prop:deintert basis}. Its proof is similarly to the degenerate case.
\begin{proposition} \label{prop:intert basis}The algebra $\mcH_{q,\mcF}$ is generated by $\{X_1,\cdots,X_n,\Phi_1,\cdots,\Phi_n,f^{-1}\ | \ 0\neq f\in \Bbbk[X]\}$ subject to the following relations for all admissible $r,s$:
	\begin{align} \label{XrXs}X_rX_s&=X_sX_r;\\
	\label{ff^{-1}} ff^{-1}&=f^{-1}f=1 \hspace{15mm} \text{$\forall 0\neq f\in \Bbbk[X]$};\\
	\label{gPhi} \Phi_rX_s&=\sigma_r(X_s)\Phi_r;\\
	\label{Phi^2}\Phi_r^2&=\tfrac{(1-qX_r)(q-X_r)}{(1-X_r)^2};\\
	\label{Phi2}\Phi_r\Phi_s&=\Phi_s\Phi_r \hspace{20mm} \text{if $r\nslash s$};\\
	\label{Phi3}\Phi_r\Phi_s\Phi_r&=\Phi_s\Phi_r\Phi_s \hspace{16.4mm} \text{if $r-\-s$};\\
	\label{Phi4}(\Phi_r\Phi_s)^2&=(\Phi_s\Phi_r)^2\hspace{15mm} \text{if $\xymatrix@C10pt{r\ar@{=}[r]|{\rangle}\ar@{=}[r]&s}$};\\
	\label{Phi6}(\Phi_r\Phi_s)^3&=(\Phi_s\Phi_r)^3 \hspace{15mm} \text{if $\xymatrix@C10pt{r\ar@3{-}[r]|{\rangle}\ar@3{-}[r]&s}$}.
	\end{align}	
\end{proposition}

 If $w=\sigma_{r_1} \sigma_{r_2}\cdots \sigma_{r_m}$ is a reduced expression in $\mcW$, then we have a well-defined element $$\Phi_w:=\Phi_{r_1}\Phi_{r_2}\cdots \Phi_{r_m}\in \mcH_{q,\mcF}.$$ The Proposition above shows that $\mcH_{q,F}$ has decompositions
\begin{equation} \label{rationalization Tw}
\mcH_{q,F}=\oplus_{w\in \mcW}\Bbbk(X)T_w=\oplus_{w\in \mcW}T_w\Bbbk(X)
\end{equation}
and \begin{equation} \label{rationalization phi_w}
\mcH_{q,F}=\oplus_{w\in \mcW}\Bbbk(X)\Phi_w=\oplus_{w\in \mcW}\Phi_w\Bbbk(X).
\end{equation}

\subsection{The Lusztig extension of $\mcH_q$} Let $\mathcal{E}$ be the unital $\Bbbk$-algebra as defined in (\ref{deepsilon}).
The {\it Lusztig extension} of $\mcH_q$ with respect to $\mathcal{E}$ is the $\Bbbk$-algebra $\mathcal{L}_q$ which is equal as $\Bbbk$-space to the tensor product $$\mathcal{L}:=\mcH_{q,F}\otimes_\Bbbk\mathcal{E}=\oplus_{w\in \mcW,\bi\in \mcC}\Phi_w\Bbbk(X)\epsilon(\bi)$$
of the rationalization algebra $\mcH_{q,F}$ and the semi-simple algebra $\mathcal{E}$. Multiplication is defined so that $\mcH_{q,F}$ (identified with the subspace $\mcH_{q,F}\otimes 1$) and $\mathcal{E}$ (identified with the subspace $1\otimes \mathcal{E}$) are subalgebras of $\mathcal{L}_q$, and in addition
\begin{align}
\label{Xepis}X_r\epsilon(\bi)&=\epsilon(\bi)X_r,\\
\label{Phiepsi}\Phi_r\epsilon(\bi)&=\epsilon(\sigma_r(\bi))\Phi_r.
\end{align}

For $r\in [n]$ and $\bi\in \mcC$, it is easy to see that
\begin{equation} \label{Trepsilon}T_r\epsilon(\bi)=\epsilon(\sigma_r(\bi))T_r+(1-q)(1-X_r)^{-1}\epsilon(\sigma_r(\bi))-(1-q)(1-X_r)^{-1}\epsilon(\bi).\end{equation}

Similar to the degenerate case, the algebra $\mathcal{L}_q$ has decompositions
\begin{equation} \label{Lusztig Tw}
\mathcal{L}_q=\oplus_{w\in \mcW,\bi\in \mcC}\Bbbk(X)T_w\epsilon(\bi)=\oplus_{w\in \mcW,\bi\in \mcC}T_w\Bbbk(X)\epsilon(\bi)
\end{equation}
and \begin{equation} \label{Lusztig Phi_w}
\mathcal{L}_q=\oplus_{w\in \mcW,\bi\in \mcC}\Bbbk(X)\Phi_w\epsilon(\bi)=\oplus_{w\in \mcW,\bi\in \mcC}\Phi_w\Bbbk(X)\epsilon(\bi).
\end{equation}

\subsection{Brundan-Kleshchev auxiliary elements}
Recall that \cite[(4.21)]{Brundan-Kleshchev09}, for each $r\in [n]$, Brundan and Kleshchev introduced the elements
\begin{equation} \label{YX} Y_r:=\sum_{\bi\in \mcC}(1-q^{-i_r}X_r)\epsilon(\bi)\end{equation} Then $Y_r\epsilon(\bi)=\epsilon(\bi)Y_r$ by \ref{Xepis}, and $Y_r$ is a unit in $\mathcal{L}_q$ with $$Y_r^{-1}=\sum_{\bi\in \mcC}(1-q^{-i_r}X_r)^{-1}\epsilon(\bi).$$

Let $\Bbbk[Y^{\pm 1}]$ be the Laurent polynomial ring with $Y:=Y_1,Y_2,\cdots,Y_n$ and $\Bbbk(Y)$ be the rational function field. Similar to the proof of Lemma \ref{lem:K(x)=K(y)}, we know that $\Bbbk(X)\otimes_\Bbbk\mathcal{E}=\Bbbk(Y)\otimes_\Bbbk\mathcal{E}$ in $\mathcal{L}_q$.

We observe that there is an action of $\mcW$ on $\Bbbk[Y^{\pm 1}]$ (by the ring automorphism) since $$\sigma_r(Y_s)=\sum_{\bi\in\mcC}(\sigma_r(X_s)-q^{i_s})\epsilon(\sigma_r(\bi))=1-(1-Y_s)(1-Y_r)^{-a_{sr}}$$
for all $r,s\in[n]$. It can be extended to an action of $\mcW$ on the rational function field $\Bbbk(Y)$ via the field automorphism
$w(\frac{f(Y)}{g(Y)})=\frac{f(w(Y))}{g(w(Y))}$ for any $w\in \mcW$. 
For simplicity, for each $r\in [n]$, we shall write $Y_{r1}=\sigma_s(Y_r), Y_{r2}=\sigma_r\sigma_s(Y_r), Y_{r3}=\sigma_s\sigma_r\sigma_s(Y_r), Y_{r4}=\sigma_r\sigma_s\sigma_r\sigma_s(Y_r), Y_{r5}=\sigma_s\sigma_r\sigma_s\sigma_r\sigma_s(Y_r)$ in the sequel.

Following Brundan and Kleshchev, for each $r\in [n]$, we define the element
\begin{equation}\label{Qr(i)} Q_r(\bi)=\begin{cases} 1-q-Y_r & \text{if $i_r=0$},\\
1 & \text{if $i_r=1, e\neq 2$},\\
[q-q^{-1}(1-Y_r)][1-q^{-1}(1-Y_r)]^{-2}& \text{if $i_r=-1, e\neq 2$},\\
[q(1-Y_r)-1]^{-1} & \text{if $i_r=1, e=2$},\\
[q^{i_r}(1-Y_r)-q][q^{i_r}(1-Y_r)-1]^{-1}& \text{if $i_r\neq 0, \pm 1$}.
\end{cases}
\end{equation}
Similar to the proof of \cite[(4.33)-(4.35)]{Brundan-Kleshchev09}, we can deduce the following result.
\begin{lemma} Let $r,s\in [n]$ and $\bi\in \mcC$. Then
	\begin{equation}\label{QsrQsr} \sigma_r(Q_r(\sigma_r(\bi)))Q_r(\bi)=\begin{cases}\tfrac{[(1-Y_r)-q][1-q(1-Y_r)]}{1-y_r} & \text{if $i_r=0$};\\
	\tfrac{q(Y_r-1)[1-q^2(1-Y_r)]}{[1-q(1-Y_r)]^{2}}&\text{if $i_r=1, e\neq 2$};\\
	\tfrac{q-q^{-1}(1-Y_r)}{[1-q^{-1}(1-Y_r)]^2}&\text{if $i_r=-1, e\neq2$};\\
	-\tfrac{q(1-Y_r)}{[1-q(1-Y_r)]^2}&\text{if $i_r=1, e=2$};\\
	\tfrac{[q-q^{i_r}(1-Y_r)][1-q^{1+i_r}(1-Y_r)]}{[1-q^{i_r}(1-Y_r)]^2}& \text{if $i_r\neq 0, \pm 1$}.
	\end{cases}
	\end{equation}
\end{lemma}
Inside $\mathcal{L}_q$, set $$\Theta_r=\Phi_r\sum_{\bi\in \mcC}Q_r^{-1}(\bi)\epsilon(\bi).$$
 These elements have the following nice properties by using relations (\ref{XrXs})-(\ref{Phi6}):
	\begin{align}\label{Thetaepsilon} \Theta_r\epsilon(\bi)&=\epsilon(\sigma_r(\bi))\Theta_r;\\
\label{fTheta} Y_s\Theta_r&=\Theta_r\sigma_r(Ys);\\
\label{Theta2}\Theta_r\Theta_s&=\Theta_s\Theta_r \hspace{19mm} \text{if $r\nslash s$};\\
\label{Theta3}\Theta_r\Theta_s\Theta_r&=\Theta_s\Theta_r\Theta_s \hspace{16mm} \text{if $r-\-s$};\\
\label{Theta4}(\Theta_r\Theta_s)^2&=(\Theta_s\Theta_r)^2\hspace{15mm} \text{if $\xymatrix@C10pt{r\ar@{=}[r]|{\rangle}\ar@{=}[r]&s}$};\\
\label{Theta6}(\Theta_r\Theta_s)^3&=(\Theta_s\Theta_r)^3 \hspace{15mm} \text{if $\xymatrix@C10pt{r\ar@3{-}[r]|{\rangle}\ar@3{-}[r]&s}$}
	\end{align}
	and \begin{equation}\label{Theta^2}\Theta_r^2\epsilon(\bi)=\begin{cases}(Y_r-1)Y_r^{-2}\epsilon(\bi) & \text{if $i_r=0$},\\
	Y_r(Y_r-1)^{-1}\epsilon(\bi) &\text{if $i_r=1, e\neq 2$},\\
	Y_r\epsilon(\bi) & \text{if $i_r=-1, e\neq 2$},\\
	Y_r^2(Y_r-1)^{-1}\epsilon(\bi)& \text{if $i_r=1, e=2$},\\
	\epsilon(\bi)& \text{if $i_r\neq0,\pm1$}.
	\end{cases}
	\end{equation}

In $\mathcal{L}_q$, we have a well-defined element $\Theta_w:=\Theta_{r_1}\Theta_{r_2}\cdots \Theta_{r_m}$ if $w=\sigma_{r_1} \sigma_{r_2}\cdots \sigma_{r_m}$ is a reduced expression in $\mcW$. Since $\Phi_r=\Theta_r\sum_{\bi\in \mcC} Q_r(\bi)\epsilon(\bi)$, thus by the decompositions (\ref{Lusztig Phi_w}) and $\Bbbk(X)\otimes_\Bbbk\mathcal{E}=\Bbbk(Y)\otimes_\Bbbk\mathcal{E}$, we obtain that
\begin{equation} \label{Lusztig Theta_w}
\mathcal{L}_q=\oplus_{w\in \mcW,\bi\in \mcC}\Bbbk(Y)\Theta_w\epsilon(\bi)=\oplus_{w\in \mcW,\bi\in \mcC}\Theta_w\Bbbk(Y)\epsilon(\bi).
\end{equation}

\subsection{The KLR-type generators of $\mathcal{L}_q$} In $\mathcal{L}_q$, for each $r\in [n]$, set \begin{equation} \label{Psitheta}\Psi_r=\sum_{\bi\in \mcC}[\Theta_r-\delta_{i_r}^{0}Y_r^{-1}]\epsilon(\bi).
\end{equation}
Using these elements, we can construct a KLR-type generators of $\mathcal{L}_q$.
\begin{theorem} \label{thm:KLR basis} The algebra $\mathcal{L}_q$ is generated by $\{Y_1,\cdots,Y_n,\Psi_1,\cdots,\Psi_n,f^{-1}, \epsilon(\bi)\ | \ 0\neq f\in \Bbbk[Y], \bi\in \mcC\}$ subject to the following relations for all admissible indices:
	
	$(1)$ $\epsilon(\bi)\epsilon(\bj)=\delta_{\bi}^{\bj}\epsilon(\bi), \quad \sum_{\bi\in \mcC}\epsilon(\bi)=1.$
	
	$(2)$  $Y_r\epsilon(\bi)=\epsilon(\bi)Y_r, Y_rY_s=Y_sY_r.$
	
	$(3)$ For any $0\neq f\in \Bbbk[Y]$, there holds that $ff^{-1}=f^{-1}f=1.$
	
	$(4)$  $ \Psi_r\epsilon(\bi)=\epsilon(\sigma_r(\bi))\Psi_r.$
	
	$(5)$ $\Psi_rY_s\epsilon(\bi)=[\sigma_r(Y_s)\psi_r+\delta_{i_r}^{0}\parti_r(Y_s)]\epsilon(\bi).$

	$(6)$ $\Psi_r^2\epsilon(\bi)=\begin{cases}-\Psi_r\epsilon(\bi) & \text{if $i_r=0$},\\		
	Y_r(Y_r-1)^{-1}\epsilon(\bi) &\text{if $i_r=1, e\neq 2$},\\
	Y_r\epsilon(\bi) & \text{if $i_r=-1, e\neq 2$},\\
	Y_r^2(Y_r-1)^{-1}\epsilon(\bi)& \text{if $i_r=1, e=2$},\\
	\epsilon(\bi)& \text{if $i_r\neq0,\pm1$}.
	\end{cases}$
	
	$(7)$  If $r\nslash s$, then $\Psi_r\Psi_s=\Psi_s\Psi_r.$
	
	$(8)$ If $r-\-s$, then \begin{equation*} \label{Psi3}(\Psi_r\Psi_s\Psi_r-\Psi_s\Psi_r\Psi_s)\epsilon(\bi)=\begin{cases}\frac{1}{1-Y_r}\epsilon(\bi) & \text{if $i_r=-i_s=1, e\neq 2$},\\
	\frac{1}{Y_s-1}\epsilon(\bi) &\text{if $i_r=-i_s=-1, e\neq 2$},\\
	\frac{Y_r-Y_s}{(1-Y_r)(1-Y_s)}\epsilon(\bi) & \text{if $i_r=i_s=1, e=2$},\\
	0 & \text{else}.
	\end{cases}
	\end{equation*}
	
	$(9)$ If $\xymatrix@C10pt{r\ar@{=}[r]|{\rangle}\ar@{=}[r]&s}$, then
	\begin{equation*}\label{Psi4} [(\Psi_r\Psi_s)^2-(\Psi_s\Psi_r)^2]\epsilon(\bi)=\begin{cases}\tfrac{1}{Y_s-1}\Psi_r\epsilon(\bi) & \text{if $i_s=1, i_r=-2, e\neq 2$};\\
	\frac{1}{1-Y_{s1}}\Psi_r\epsilon(\bi)&\text{if $i_s=-1, i_r=2,e\neq 2$}; \\
	\frac{1}{1-Y_{s1}}(Y_r\Psi_r+1)\epsilon(\bi) & \text{if $i_r=0,i_s=1,e=2$},\\
	\frac{2-Y_{s1}}{1-Y_r}\Psi_s\epsilon(\bi) & \text{if $i_r=-i_s=1,e\neq 2$};\\
	\frac{2-Y_{s1}}{Y_{r1}-1}\Psi_s \epsilon(\bi)&\text{if $i_r=-i_s=-1, e\neq 2$}; \\
	\frac{Y_s(2-Y_s)(Y_{s1}-2)}{1-Y_{r1}}\Psi_s\epsilon(\bi) & \text{if $i_r=i_s=1,e=2$};\\
	0 & \text{else}.
	\end{cases}
	\end{equation*}
	
	$(10)$ If $\xymatrix@C10pt{r\ar@3{-}[r]|{\rangle}\ar@3{-}[r]&s}$, then
	\begin{align*}\label{Psi6}
	&[(\Psi_r\Psi_s)^3-(\Psi_s\Psi_r)^3]\epsilon(\bi)\notag\\
	&=\begin{cases}\frac{1}{Y_s-1}\Psi_r\Psi_s\Psi_r\epsilon(\bi) & \text{if $i_{r}=-3, i_s=1,e\neq 2,3$},\\
	\frac{1}{1-Y_{s2}}\Psi_r\Psi_s\Psi_r\epsilon(\bi) & \text{if $i_r=3, i_s=-1, e\neq 2,3$},\\
	\frac{3-3Y_{s1}+Y_{s1}^2}{1-Y_r}\Psi_s\Psi_r\Psi_s\epsilon(\bi) & \text{if $i_r=-i_s=1,e\neq 2$},\\
	\frac{3-3Y_{s1}+Y_{s1}^2}{Y_{r2}-1}\Psi_s\Psi_r\Psi_se(\bi)\epsilon(\bi) & \text{if $i_r=-i_s=-1,e\neq 2$},\\
	\frac{(2-Y_{s2})(Y_{s}-Y_{s1}+Y_{r1}Y_{s1})}{(1-Y_{r1})(1-Y_s)}\Psi_r\epsilon(\bi) & \text{if $i_s=1,i_{r}=-2,e\neq 2$},\\
	\frac{(2-Y_{s2})(Y_{r2}-Y_{r1}+Y_{s1}Y_{r1})}{(1-Y_{r2})(1-Y_{s1})}\Psi_r\epsilon(\bi) & \text{if $i_s=-1,i_{r}=2,e\neq 2$},\\
	\frac{Y_s(Y_s-2)}{1-Y_{s2}}\Psi_s\epsilon(\bi) & \text{if $i_r=1,i_{s}=2, e=4$},\\
	\frac{Y_s(Y_s-2)}{1-Y_{r1}}\Psi_s\epsilon(\bi) & \text{if $i_r=-1,i_{s}= 2, e=4$},\\
	\frac{1}{Y_{s1}-1}\Psi_s\epsilon(\bi) & \text{if $i_r=3,i_{s}=-2, e\neq 2,3$},\\
	\frac{1}{1-Y_{s2}}\Psi_s\epsilon(\bi) & \text{if $i_r=-3,i_{s}=2, e\neq 2,3$},\\
	\frac{1}{Y_r-1}\Psi_s\epsilon(\bi) & \text{if $i_r=1,i_{s}\neq -2, e\neq 2$},\\
	\frac{1}{1-Y_{r1}}\Psi_s\epsilon(\bi) & \text{if $i_r=-1,i_{s}\neq -2, e\neq 2$},\\
	\frac{(Y_{s2}-2)[Y_{r2}Y_{s1}+Y_{r1}Y_s(1-Y_r)]}{(1-Y_{s1})(1-Y_{r2})}(Y_r\psi_r+1)\epsilon(\bi) & \text{if $i_r=0,i_s=1, e=2$},\\
	\frac{(2-Y_s)[Y_{r1}Y_{s2}+Y_rY_{s1}(1-Y_s)^2]}{(1-Y_{r1})(1-Y_{s2})}(Y_s\psi_s+1)\epsilon(\bi) & \text{if $i_r=1, i_s=0, e=2$},\\
	[\frac{Y_{s1}}{1-Y_{s2}}\Psi_r\Psi_s\Psi_r-\frac{Y_{r1}(3-3Y_{s1}+Y_{s1}^2)}{1-Y_{r2}}\Psi_s\Psi_r\Psi_s\\
	-\frac{Y_r^2}{(1-Y_{s1})^2}+\frac{Y_s^2(3-3Y_{s1}+Y_{s1}^2)}{(1-Y_{s2})^2}]\epsilon(\bi) & \text{if $i_r=i_s=1, e=2$},\\
	\frac{1}{Y_s-1}\Psi_r\Psi_s\Psi_r\epsilon(\bi) & \text{if $i_r=0,i_s=1, e=3$},\\
	\frac{1}{1-Y_{s2}}(\Psi_r\Psi_s\Psi_r+\Psi_r\Psi_s)\epsilon(\bi) & \text{if $i_r=0,i_s=-1, e=3$},\\
	0 & \text{else}.
	\end{cases}
	\end{align*}
\end{theorem}

\begin{proof} The proofs of Statements (1)-(9) are similarly to the corresponding Statements of Theorem \ref{thm:deKLR basis} in the degenerate case, hence we skip them. Here we only give the proof of the Statement (10) which is the most complicated one.	
		
	(10) Since $\xymatrix@C10pt{r\ar@3{-}[r]|{\rangle}\ar@3{-}[r]&s}$, by Statement (3), relations (\ref{Psitheta}) and (\ref{fTheta}), we get that	
	\begin{align*}
	&(\Psi_r\Psi_s)^3\epsilon(\bi)\\
	&=(\Theta_r-\tfrac{\delta_{i_{r5}}^0}{Y_r})(\Theta_s-\tfrac{\delta_{i_{s4}}^0}{Y_s})(\Theta_r-\tfrac{\delta_{i_{r3}}^0}{Y_r})(\Theta_s
	-\tfrac{\delta_{i_{s2}}^0}{Y_s})(\Theta_{r}-\tfrac{\delta_{i_{r1}}^0}{Y_r})(\Theta_s-\tfrac{\delta_{i_{s}}^0}{Y_s})\epsilon(\bi)\\
	&=\{(\Theta_r\Theta_s)^3-\tfrac{\delta_{i_s}^0}{Y_{s}}(\Theta_r\Theta_s)^2\Theta_r-\tfrac{\delta_{i_r}^0}{Y_r}\Theta_s(\Theta_r\Theta_s)^2+\tfrac{\delta_{i_{r}}^0\delta_{i_{s}}^0}{Y_{r1}Y_s}(\Theta_r\Theta_s)^2+\tfrac{\delta_{i_r}^0\delta_{i_s}^0}{Y_rY_{s1}}(\Theta_s\Theta_r)^2-(\tfrac{\delta_{i_{r}}^0\delta_{i_s}^0}{Y_sY_{s2}^{2}}\\
	&+\tfrac{\delta_{i_{r1}}^0\Theta_{s2}^{2}}{Y_{r1}})\Theta_r\Theta_s\Theta_r-(\tfrac{\delta_{i_{s1}}^0\Theta_r^2}{Y_{s1}}+\tfrac{\delta_{i_{r}}^0\delta_{i_{s}}^0(3-3Y_{s1}+Y_{s1}^2)}{Y_r^2Y_{r2}})\Theta_s\Theta_r\Theta_s+\delta_{i_{r}}^0[\delta_{i_{s}}^0(\tfrac{\Theta_{r2}^2}{\sigma_r\sigma_s(Y_{s})Y_{s2}}+\tfrac{\sigma_r(\Theta_{r}^{2})}{Y_sY_{s1}}\\
	&+\tfrac{(Y_{s1}-1)^2}{Y_{r2}^{2}Y_{s1}^{2}}+\tfrac{3-Y_{s}-Y_{s2}}{\sigma_r(Y_r^{2})Y_{r2}Y_{s1}}
	+\tfrac{2-Y_{s2}}{Y_rY_{r2}Y_s^{2}})+\delta_{i_{r2}}^0(\tfrac{\sigma_r\sigma_s(\Theta_s^2)}{\sigma_r(Y_r)Y_{r2}}+\tfrac{\Theta_s^2}{Y_rY_{r1}})]\Theta_r\Theta_s+\delta_{i_r}^0[\delta_{i_{s}}^0(\tfrac{\Theta_r^2}{Y_{s1}Y_{s2}}+\tfrac{2-Y_{s1}}{Y_r^2Y_{s2}^{2}})\\
	&+\tfrac{\delta_{i_{r1}}^0\Theta_{s2}^{2}}{Y_rY_{r2}})]\Theta_s\Theta_r-[\delta_{i_{r}}^0\delta_{i_s}^0(\tfrac{(3-Y_{s}-Y_{s2})\Theta_{s1}^2}{\sigma_r(Y_r)Y_{r2}^{2}}+\tfrac{\Theta_{s1}^2}{Y_{r2}Y_sY_{s1}}+\tfrac{\sigma_r(\Theta_{r}^2)}{Y_{s}Y_{s1}^{2}}+\tfrac{\Theta_s^2}{Y_rY_{r1}Y_{s1}}+\tfrac{1}{\sigma_r(Y_r^{2})Y_{s1}^{3}}+\tfrac{1}{Y_rY_s^{2}Y_{s1}^{2}})\\
	&+\tfrac{\delta_{i_{s2}}^0\Theta_{r2}^{2}\Theta_{s1}^2}{Y_{s2}}]\Theta_r-[\delta_{i_r}^0\delta_{i_{s}}^0(\tfrac{(2-Y_{s2})\Theta_r^2}{Y_{r1}Y_{s1}^{2}}+\tfrac{\Theta_{r}^{2}}{Y_rY_sY_{s1}}+\tfrac{\Theta_{r1}^2}{Y_r\sigma_s(Y_s)Y_{s2}}+\tfrac{1}{Y_rY_{r1}^{2}\sigma_s(Y_s^{2})}+\tfrac{3-Y_{s1}-Y_{s2}}{Y_r^{3}Y_{r1}Y_s})\\
	&+\tfrac{\delta_{i_{r2}}^0\Theta_{r}^2\Theta_{s1}^2}{Y_{r2}}+\tfrac{\delta_{i_r}^0\delta_{i_{r1}}^0\delta_{i_{r2}}^0\sigma_s(\Theta_s^2)}{Y_r^2Y_{r1}}]\Theta_s+(\tfrac{\delta_{i_{r2}}^0\delta_{i_s}^0\Theta_{s1}^2}{Y_{r2}Y_{s}}+\tfrac{\delta_{i_{r1}}^0\delta_{i_{s1}}^0\Theta_{s}^{2}}{Y_{r1}Y_{s1}})\Theta_r^2+\tfrac{\delta_{i_r}^0\delta_{i_{s2}}^0\Theta_{r1}^{2}\Theta_s^2}{Y_rY_{s2}}\\	
	&+\delta_{i_{r}}^0\delta_{i_s}^0(\tfrac{(3-Y_{s1}-Y_{s2})\Theta_s^2}{Y_r^2Y_{r1}^{2}}+\tfrac{\Theta_s^2}{Y_r^2Y_{r1}Y_s}+\tfrac{(Y_r+Y_{s1})\Theta_{r}^2}{Y_rY_s^{2}Y_{s1}^2}+\tfrac{1}{Y_r^{3}Y_s^{3}})\}\epsilon(\bi)\\
	\end{align*}	
	Similarly, we can show that
	\begin{align*}
	&(\Psi_s\Psi_r)^3\epsilon(\bi)\\
	&=\{(\Theta_s\Theta_r)^3-\tfrac{\delta_{i_{s}}^0}{Y_s}(\Theta_r\Theta_s)^2\Theta_r-\tfrac{\delta_{i_r}^0}{Y_{r}}(\Theta_s\Theta_r)^2\Theta_s+\tfrac{\delta_{i_r}^0\delta_{i_{s}}^0}{Y_{r1}Y_s}(\Theta_r\Theta_s)^2+\tfrac{\delta_{i_{r}}^0\delta_{i_{s}}^0}{Y_rY_{s1}}(\Theta_s\Theta_r)^2\\
	&-(\tfrac{\delta_{i_{r}}^0\delta_{i_{s}}^0}{Y_s^{2}Y_{s2}}+\tfrac{\delta_{i_{r1}}^0\Theta_s^2}{Y_{r1}})\Theta_r\Theta_s\Theta_r-(\tfrac{\delta_{i_r}^0\delta_{i_{s}}^0(3-3Y_{s1}+Y_{s1}^2)}{Y_rY_{r2}^{2}}+\tfrac{\delta_{i_{s1}}^0\Theta_{r2}^{2}}{Y_{s1}})\Theta_s\Theta_r\Theta_s+\delta_{i_r}^0[\delta_{i_{s}}^0(\tfrac{2-Y_{s2}}{Y_{r2}^{2}Y_s^2}\\
	&+\tfrac{\Theta_{r2}^{2}}{Y_sY_{s2}})+\tfrac{\delta_{i_{r1}}^0\Theta_s^2}{Y_{r1}Y_{r2}}]\Theta_r\Theta_s+\delta_{i_r}^0[\delta_{i_{r2}}^0(\tfrac{\Theta_{s2}^2}{\sigma_s\sigma_r(Y_r)Y_{r2}}+\tfrac{\sigma_s(\Theta_{s}^{2})}{Y_{r}Y_{r1}})+\delta_{i_{s}}^0(\tfrac{\sigma_s\sigma_r(\Theta_r^2)}{\sigma_s(Y_{s})Y_{s2}}+\tfrac{1}{\sigma_s\sigma_r(Y_r^{2})Y_{s2}^{2}}\\
	&+\tfrac{1}{Y_{r1}\sigma_s(Y_s^{2})Y_{s2}}+\tfrac{2-Y_{s1}}{Y_r^{2}Y_sY_{s2}}+\tfrac{\Theta_r^2}{Y_sY_{s1}})]\Theta_s\Theta_r-[\delta_{i_{r}}^0\delta_{i_{s}}^0(\tfrac{(2-Y_{s2})\Theta_s^2}{Y_{r1}^{2}Y_{s1}}+\tfrac{\Theta_{s1}^2}{\sigma_r(Y_r)Y_{r2}Y_s}+\tfrac{\Theta_{s}^{2}}{Y_rY_{r1}Y_s}\\
	&+\tfrac{\sigma_r(\Theta_r^2)}{Y_{s}^2Y_{s1}}+\tfrac{1}{\sigma_r(Y_r^{2})Y_sY_{s1}^{2}}+\tfrac{1}{Y_rY_s^{3}Y_{s1}})+\tfrac{\delta_{i_{s2}}^0\Theta_{r1}^2\Theta_s^2}{Y_{s2}}]\Theta_r-[\delta_{i_r}^0\delta_{i_{s}}^0(\tfrac{(Y_{r1}+Y_{s2})\Theta_{r1}^2}{Y_{r1}\sigma_s(Y_s)Y_{s2}^2}+\tfrac{1}{Y_{r1}^{3}\sigma_s(Y_s^{2})}\\
	&+\tfrac{\Theta_r^2}{Y_{r1}Y_sY_{s1}}+\tfrac{3-Y_{s1}-Y_{s2}}{Y_r^2Y_{r1}^{2}Y_s}+\tfrac{\Theta_{r1}^{2}}{Y_rY_sY_{s2}})+\tfrac{\delta_{i_{r2}}^0\Theta_{s2}^{2}\Theta_{r1}^2}{Y_{r2}}+\tfrac{\delta_{i_r}^0\delta_{i_{r1}}^0\delta_{i_{r2}}^0\sigma_s(\Theta_{s}^2)}{T_{r}Y_{r1}^{2}}]\Theta_s+[\tfrac{\delta_{i_r}^0\delta_{i_{s2}}^0\Theta_{r1}^2}{Y_{r}Y_{s2}}\\
	&+\tfrac{\delta_{i_{r1}}^0\delta_{i_{s1}}^0\Theta_{r}^{2}}{Y_{r1}Y_{s1}}]\Theta_s^2+\tfrac{\delta_{i_{r2}}^0\delta_{i_{s}}^0\Theta_{s1}^{2}\Theta_r^2}{Y_{r2}Y_s}+\delta_{i_r}^0\delta_{i_{s}}^0[\tfrac{(Y_r+Y_{s1})\Theta_r^2}{Y_rY_s^2Y_{s1}^2}+\tfrac{(3-Y_{s1}-Y_{s2})\Theta_s^2}{Y_{r}^2Y_{r1}^{2}}+\tfrac{1}{Y_r^{3}Y_s^{3}}+\tfrac{\Theta_{s}^2}{Y_r^{2}Y_{r1}Y_s}]\}\epsilon(\bi)
	\end{align*}	
	Using relations  (\ref{Theta6}), (\ref{Theta^2}) and (\ref{Psitheta}), we obtain that
	\begin{align*}
	&[(\Psi_r\Psi_s)^3-(\Psi_s\Psi_r)^3]\epsilon(\bi)\\
	&=\{(\tfrac{\delta_{i_{r}}^0\delta_{i_s}^0(1-Y_s)Y_{s1}}{Y_s^{2}Y_{s2}^{2}}+\tfrac{\delta_{i_{r1}}^0\Theta_s^2-\Theta_{s2}^2}{Y_{r1}})\Theta_r\Theta_s\Theta_r+(\tfrac{\delta_{i_{r}}^0\delta_{i_s}^0(3-3Y_{s1}+Y_{s1}^2)(Y_r-1)Y_{r1}}{Y_r^{2}Y_{r2}^{2}}+\tfrac{\delta_{i_{s1}}^0(\Theta_{r2}^2-\Theta_{r}^2)}{Y_{s1}})\Theta_s\Theta_r\Theta_s\\
	&+\delta_{i_r}^0[\delta_{i_s}^0(\tfrac{3-Y_s-Y_{s2}}{\sigma_r(Y_r^2)Y_{r2}Y_{s1}}+\tfrac{(1-Y_{s1})^2}{Y_{r2}^{2}Y_{s1}^{2}}-\tfrac{(2-Y_{s2})Y_{r1}}{\sigma_r(Y_{r})Y_{r2}^{2}Y_s^{2}})+\tfrac{\delta_{i_{s}}^0\sigma_r(\Theta_r^2)-\Theta_{r2}^2}{Y_sY_{s1}}+\tfrac{\delta_{i_{r1}}^0(\sigma_r\sigma_s(\Theta_s^2)-\Theta_s^2)}{\sigma_r(Y_r)Y_{r2}}]\Theta_r\Theta_s\\
	&+\delta_{i_r}^0[\delta_{i_s}^0(\tfrac{(Y_s+Y_{s2}-2)Y_{s1}}{Y_r^{2}Y_{s}Y_{s2}^{2}}-\tfrac{(1-Y_{r1})^2}{Y_{r1}^{2}Y_{s2}^{2}}-\tfrac{(1-Y_s)^2}{Y_{r1}Y_{s}^{2}Y_{s2}})+\tfrac{\delta_{i_{r1}}^0(\Theta_{s2}^2-\sigma_s(\Theta_s^2))}{Y_rY_{r1}}+\tfrac{\delta_{i_{s}}^0(1-Y_s)(\sigma_s\sigma_r(\Theta_r^2)-\Theta_r^2)}{Y_sY_{s2}}]\Theta_s\Theta_r\\
	&+[\delta_{i_r}^0\delta_{i_{s}}^0(\tfrac{(1-Y_s)\Theta_s^2}{Y_{r1}Y_sY_{s1}}+\tfrac{(2-Y_{s2})\Theta_s^2}{Y_{r1}^{2}Y_{s1}}+\tfrac{(Y_{r2}+Y_{s2}-2)\Theta_{s1}^{2}}{Y_{r2}^{2}Y_{s}}-\tfrac{\Theta_{s1}^{2}}{Y_{r2}Y_{s}Y_{s1}}+\tfrac{(1-Y_s)Y_r\sigma_r(\Theta_{r}^2)}{Y_{s}^{2}Y_{s1}^{2}}+\tfrac{1-Y_s}{Y_s^{3}Y_{s1}^{2}}\\
	&+\tfrac{(1-Y_r)(1-Y_{s1})}{Y_rY_sY_{s1}^{3}})+\tfrac{\delta_{i_{s2}}^0(\Theta_{r1}^2\Theta_s^2-\Theta_{r2}^{2}\Theta_{s1}^2)}{Y_{s2}}]\Theta_r+[\delta_{i_r}^0\delta_{i_{s}}^0(\tfrac{(Y_{r1}+Y_{s2})\Theta_{r1}^{2}}{Y_{r1}\sigma_s(Y_s)Y_{s2}^2}+\tfrac{(2-Y_s)\Theta_{r1}^{2}}{Y_{r}Y_sY_{s2}}-\tfrac{(2-Y_{s2})\Theta_r^2}{Y_{r1}Y_{s1}^{2}}\\
	&+\tfrac{(Y_r-Y_{r1})\Theta_r^2}{Y_rY_{r1}Y_sY_{s1}}+\tfrac{Y_r-Y_{r1}}{Y_rY_{r1}^3\sigma_s(Y_s^2)}+\tfrac{(Y_r-Y_{r1})(3-Y_{s1}-Y_{s2})}{Y_r^{3}Y_{r1}^{2}Y_{s}})+\tfrac{\delta_{i_r}^0\delta_{i_{r1}}^0(Y_r-Y_{r1})\sigma_s(\Theta_{s}^2)}{Y_{r}^{2}Y_{r1}^{2}}\\
	&+\tfrac{\delta_{i_{r2}}^0(\Theta_{r1}^2\Theta_{s2}^2-\Theta_{r}^{2}\Theta_{s1}^2)}{Y_{r2}}]\Theta_s\}\epsilon(\bi)\\
	&=\begin{cases}  \delta_{i_{r1}}^0\frac{\Theta_s^2-\Theta_{s2}^2}{Y_{r1}}\Theta_r\Theta_s\Theta_r\epsilon(\bi) & \text{if $i_{r1}=0,i_r\neq 0,i_{s1}\neq 0$},\\
	\delta_{i_{s1}}^0\frac{\Theta_{r2}^2-\Theta_{r}^2}{Y_{s1}}\Theta_s\Theta_r\Theta_s\epsilon(\bi) & \text{if $i_{s1}=0,i_r\neq 0,i_{r1}\neq 0$},\\
	\delta_{i_{s2}}^0\frac{\Theta_{r1}^2\Theta_s^2-\Theta_{r2}^{2}\Theta_{s1}^2}{Y_{s2}}\Theta_r\epsilon(\bi) & \text{if $i_{s2}=0,i_r\neq 0$},\\
	\delta_{i_{r2}}^0\frac{\Theta_{r1}^2\Theta_{s2}^2-\Theta_{r}^{2}\Theta_{s1}^2}{Y_{r2}}\Theta_s\epsilon(\bi) & \text{if $i_{r2}=0,i_r\neq 0$},\\
	\delta_{i_{s2}}^0\frac{\Theta_{r1}^2\Theta_s^2-\Theta_{r2}^{2}\Theta_{s1}^2}{Y_{s2}}\Theta_r\epsilon(\bi) & \text{if $i_{s2}=0,i_r=0,i_s\neq 0$},\\
	\delta_{i_{r2}}^0\frac{\Theta_{r1}^2\Theta_{s2}^2-\Theta_{r}^{2}\Theta_{s1}^2}{Y_{r2}}\Theta_s\epsilon(\bi) & \text{if $i_{r2}=0,i_s=0,i_r\neq 0$},\\
	(\delta_{i_{r1}}^0\frac{\Theta_s^2-\Theta_{s2}^2}{Y_{r1}}\Theta_r\Theta_s\Theta_r+\delta_{i_{s1}}^0\frac{\Theta_{r2}^2-\Theta_{r}^2}{Y_{s1}}\Theta_s\Theta_r\Theta_s)\epsilon(\bi) & \text{if $i_{r1}=0,i_{s1}=0,i_{r}\neq 0$},\\
	(\delta_{i_{r1}}^0\frac{\theta_s^2-\theta_{s2}^2}{y_{r1}}\theta_r\theta_s\theta_r+\delta_{i_r}^0\delta_{i_{r1}}^0\frac{\sigma_r\sigma_s(\theta_s^2)-\theta_s^2}{\sigma_r(Y_r)Y_{r2}}\Theta_r\Theta_s\\
	+\delta_{i_r}^0\delta_{i_{r1}}^0\frac{\Theta_{s2}^2-\sigma_s(\Theta_s^2)}{Y_rY_{r1}}\Theta_s\Theta_r+\delta_{i_r}^0\delta_{i_{r1}}^0\frac{(Y_r-Y_{r1})\sigma_s(\Theta_{s}^2)}{Y_{r}^{2}Y_{r1}^{2}}\Theta_s\\
	+\delta_{i_{r2}}^0\frac{\Theta_{r1}^2\Theta_{s2}^2-\Theta_{r}^{2}\Theta_{s1}^2}{Y_{r2}}\Theta_s)\epsilon(\bi) & \text{if $i_r=0,i_{r1}=0,i_{r2}=0,i_{s}\neq 0$},\\
	0&\text{else}
	\end{cases}	\\
	&=\begin{cases}\frac{1}{Y_s-1}\Psi_r\Psi_s\Psi_r\epsilon(\bi) & \text{if $i_{r}=-3, i_s=1,e\neq 2,3$},\\
	\frac{1}{1-Y_{s2}}\Psi_r\Psi_s\Psi_r\epsilon(\bi) & \text{if $i_r=3, i_s=-1, e\neq 2,3$},\\
	\frac{3-3Y_{s1}+Y_{s1}^2}{1-Y_r}\Psi_s\Psi_r\Psi_s\epsilon(\bi) & \text{if $i_r=-i_s=1,e\neq 2$},\\
	\frac{3-3Y_{s1}+Y_{s1}^2}{Y_{r2}-1}\Psi_s\Psi_r\Psi_se(\bi)\epsilon(\bi) & \text{if $i_r=-i_s=-1,e\neq 2$},\\
	\frac{(2-Y_{s2})(Y_{s}-Y_{s1}+Y_{r1}Y_{s1})}{(1-Y_{r1})(1-Y_s)}\Psi_r\epsilon(\bi) & \text{if $i_s=1,i_{r}=-2,e\neq 2$},\\
	\frac{(Y_{s2}-2)(Y_{s1}-Y_{s}+Y_{r2}Y_{s})}{(1-Y_{r2})(1-Y_{s1})}\Psi_r\epsilon(\bi) & \text{if $i_s=-1,i_{r}=2,e\neq 2$},\\
	\frac{Y_s(Y_s-2)}{1-Y_{s2}}\Psi_s\epsilon(\bi) & \text{if $i_r=1,i_{s}=2, e=4$},\\
	\frac{Y_s(Y_s-2)}{1-Y_{r1}}\Psi_s\epsilon(\bi) & \text{if $i_r=-1,i_{s}= 2, e=4$},\\
	\frac{1}{1-Y_{s1}}\Psi_s\epsilon(\bi) & \text{if $i_r=3,i_{s}=-2, e\neq 2,3$},\\
	\frac{1}{Y_{s2}-1}\Psi_s\epsilon(\bi) & \text{if $i_r=-3,i_{s}=2, e\neq 2,3$},\\
	\frac{1}{1-Y_r}\Psi_s\epsilon(\bi) & \text{if $i_r=1,3i_{s}=-2, e\neq 2,4$},\\
	\frac{1}{Y_{r1}-1}\Psi_s\epsilon(\bi) & \text{if $i_r=-1,3i_{s}=2, e\neq 2,4$},\\
	\frac{(Y_{s2}-2)[Y_{r2}Y_{s1}+Y_{r1}Y_s(1-Y_r)]}{(1-Y_{s1})(1-Y_{r2})}(Y_r\Psi_r+1)\epsilon(\bi) & \text{if $i_r=0,i_s=1, e=2$},\\
	\frac{(2-Y_s)[Y_{r1}Y_{s2}+Y_rY_{s1}(1-Y_s)^2]}{(1-Y_{r1})(1-Y_{s2})}(Y_s\psi_s+1)\epsilon(\bi) & \text{if $i_r=1, i_s=0, e=2$},\\
	[\frac{Y_{s1}}{1-Y_{s2}}\Psi_r\Psi_s\Psi_r-\frac{Y_{r1}(3-3Y_{s1}+Y_{s1}^2)}{1-Y_{r2}}\Psi_s\Psi_r\Psi_s\\
	+\frac{Y_{r2}-Y_{r1}+Y_s-Y_{s1}+Y_sY_{r2}}{(1-Y_{s})(1-Y_{r2})}]\epsilon(\bi) & \text{if $i_r=i_s=1, e=2$},\\
	\frac{1}{Y_s-1}\Psi_r\Psi_s\Psi_r\epsilon(\bi) & \text{if $i_r=0,i_s=1, e=3$},\\
	\frac{1}{1-Y_{s2}}(\Psi_r\Psi_s\Psi_r+\Psi_r\Psi_s)\epsilon(\bi) & \text{if $i_r=0,i_s=-1, e=3$},\\
	0 & \text{else}.
	\end{cases}
	\end{align*}
		
For each $w\in \mcW$, we {\it fix} a reduced decomposition $w=\sigma_{r_1} \sigma_{r_2}\cdots \sigma_{r_m}$ and define the element
$$\Psi_w:=\Psi_{r_1}\Psi_{r_2}\cdots \Psi_{r_m}\in \mathcal{L}_q.$$
Similar to the degenerate case, we can show that $\{\Psi_w \ | \ w\in \mcW\}$ is a basis of $\mathcal{L}_q$ as $\Bbbk(Y)\otimes_\Bbbk\mathcal{E}$-module. Thus the relations (1)-(9) is complete since by them every element in $\mathcal{L}_q$ can be written as $\sum_{w\in \mcW,\bi\in \mcC}\Psi_wf_{w,\bi}(Y)\epsilon(\bi)$ with $f_{w,\bi}(Y)\in \Bbbk(Y)$. 	
\end{proof}

\subsection{The BK-subalgebras}
We define the {\it BK-subalgebra} as the $\Bbbk$-algebra $\tilde{\mathcal{L}}_q$ generated by $$\{Y_1,\cdots,Y_n, \Psi_1,\cdots,\Psi_n, f^{-1}(Y),\epsilon(\bi) \ |\ \bi\in \mcC, f(Y)\in \Bbbk[Y] \ \mbox{with} \ f(0)\neq 0\}$$
subject to the relations (1)-(10) of Theorem \ref{thm:KLR basis}.

Similar to the proof of Corollary \ref{cor:basis deBK algebra}, we have the following result.
\begin{corollary}\label{cor:basis BK algebra}  Denote by $1\epsilon(\bi)=\epsilon(\bi)$. Then the algebra $\tilde{\mathcal{L}}_q$ is generated by $$\{X_1,\cdots,X_n, T_1,\cdots,T_{n}, f^{-1}(X)\epsilon(\bi) \ |\ \bi\in \mcC, f(X)\in \Bbbk[X] \ \mbox{with} \ f(q^\bi)\neq 0\}$$
	subject to relations (\ref{affhecke X})-(\ref{TTTTTT}), (\ref{deepsilon}), (\ref{Xepis}), (\ref{Trepsilon}) and
	\begin{equation}\label{semif(X)f^{-1}(X)}
	\epsilon(\bj)\cdot f^{-1}\epsilon(\bi)=\delta_{\bi}^{\bj}f^{-1}\epsilon(\bi)=f^{-1}\epsilon(\bi)\cdot \epsilon(\bj), \ f\cdot f^{-1}\epsilon(\bi)=\epsilon(\bi)=f^{-1}\epsilon(\bi)\cdot f
	\end{equation}
	for $f\in \Bbbk[X]$ with $f(q^\bi)\neq 0$.
\end{corollary}

Denote by $$\tilde{\mathcal{L}}_q(\Lambda):=\tilde{\mathcal{L}}_q/\langle Y_1^{\Lambda_{i_1}}\epsilon(\bi) | \bi\in \mcC \rangle.$$
We use the same letters $\Psi_1,\cdots,\Psi_n$ and $Y_1,\cdots,Y_n$ to denote the images of the generators in $\tilde{\mathcal{L}}_q(\Lambda)$. Imitate the proof of Lemma \ref{lem:cdesemiration}, we can deduce the following result.

\begin{lemma} \label{lem:csemiration} $\tilde{\mathcal{L}}_q(\Lambda)=\tilde{\mathcal{L}}_q/\langle \prod_{i\in I}(X_1-q^i)^{\Lambda_i}\rangle.$
\end{lemma}
\begin{remark} Similar to the degenerate case, the elements $Y_r, \prod_{i\in I}(X_r-q^i)$ are all nilpotent in $\tilde{\mathcal{L}}_q(\Lambda)$.
\end{remark}

\subsection{The cyclotomic non-degenerate affine Hecke algebras}

We define the {\itshape non-degenerate cyclotomic affine Hecke algebra $\mcH_q(\Lambda)$} as
$$\mcH_q(\Lambda):=\mcH_q/\langle\prod_{i\in I}(X_1-q^i)^{\Lambda_i}\rangle.$$
Similar to \cite[Subsection 4.1]{Brundan-Kleshchev09}, there is a system $\{e(\bi) \ | \ \bi\in \mcC\}$ of mutually orthogonal idempotents in $\mcH_q(\Lambda)$ such that $1=\sum_{\bi\in I^n}e(\bi)$ and \begin{equation*}e(\bi)\mcH_q(\Lambda)=\{h\in \mcH_q(\Lambda) \ | \ (X_r-q^{i_r})^mh=0 \ \mbox{for all} \ r\in [n] \ \mbox{and} \ m\gg0\}.\end{equation*}
It is easy to see that $X_re(\bi)=e(\bi)X_r$ for all $r\in [n]$ and $\bi\in \mcC$, and for a polynomial $f(X)\in \Bbbk[X]$, $f(X)e(\bi)$ is a unit if and only if $f(q^{\bi})\neq 0$. In particular, the element $(1-X_r)e(\bi)$ is a unit in $e(\bi)\mcH_q(\Lambda)$ if and only if $i_r\neq 0$. In this case, we write $(1-X_r)^{-1}e(\bi)$ for the inverse.

\begin{lemma} For $r\in [n]$ and $\bi\in I^{n}$, there holds that
	\begin{equation} \label{Tre-eTr} T_re(\bi)=\begin{cases}
	e(\bi)T_r & \text{if $i_r=0$},\\
	e(\sigma_r(\bi))T_r+\tfrac{1-q}{1-X_r}e(\sigma_r(\bi))-\tfrac{1-q}{1-X_r}e(\bi) & \text{if $i_r\neq 0$}.
	\end{cases}
	\end{equation}
\end{lemma}
\begin{proof}   For any $s\in [n]$, the element $$(\sigma_r(X_s)-q^{\sigma(\bi)_s})e(\bi)=[X_r^{-a_{sr}}(X_s-q^{i_s})+(X_r-q^{i_r})f_{a_{sr}}(X_r^{\pm 1})]e(\bi)$$ for some polynomial $f_{a_{sr}}(X_r^{\pm 1})$. Thus it is nilpotent by the nilpotency of $(X_s-q^{i_s})e(\bi)$ and $(X_r-q^{i_r})e(\bi)$. Similarly, we can show that
	$\mathrm{D}_r((X_s-q^{i_s})^m)e(\bi)$ is nilpotent an integer $m\gg0$ whenever $i_r=0$. Therefore, if $i_r=0$, by (\ref{fT}), there holds that
	\begin{align*}&(X_s-q^{i_s})^mT_re(\bi)\\
	&=T_r(\sigma_r(X_s)-q^{i_s})^me(\bi)+(1-q)\mathrm{D}_r((X_s-q^{i_s})^m)e(\bi)\\
	&=0\end{align*}
	when $m\gg0$. Therefore $T_re(\bi)\in e(\bi)\mcH_q(\Lambda)$ and then $$T_re(\bi)=e(\bi)T_re(\bi)=e(\bi)T_r.$$
	
	If $i_{r}\neq 0$, by (\ref{fT}), we get
	\begin{align*}&(X_s-q^{\sigma_r(\bi)_s})^m[T_r(1-X_r)+1-q]e(\bi)\\
	&=[T_r(1-X_r)+1-q](\sigma_r(X_s)-q^{\sigma_r(\bi)_s})^me(\bi)\\
	&=0\end{align*}
	when $m\gg0$. Therefore
	\begin{align*}T_r(1-X_r)e(\bi)+(1-q)e(\bi)&=e(\sigma_r(\bi))[T_r(1-X_r)e(\bi)+(1-q)e(\bi)]\\
	&=e(\sigma(\bi))T_r(1-X_{r})e(\bi).
	\end{align*}
	Then right-multiplying by $(1-X_r)^{-1}e(\bi)$, we have $$T_re(\bi)=e(\sigma_r(\bi))T_re(\bi)-(1-q)
	(1-X_r)^{-1}e(\bi).$$
	Similarly, we can deduce that
	$$e(\sigma_r(\bi))T_r=e(\sigma_r(\bi))T_re(\bi)-(1-q)
	(1-X_r)^{-1}e(\sigma_r(\bi)).$$
	Therefore $T_re(\bi)=e(\sigma_r(\bi))T_r+\tfrac{1-q}{1-X_r}e(\sigma_r(\bi))-\tfrac{1-q}{1-X_r}e(\bi).$
\end{proof}

Let $$e(\mcC):=\sum_{\bi\in \mcC}e(\bi)\in \mcH_q(\Lambda).$$
Then $e(\mcC)$ is a central idempotent in $\mcH_q(\Lambda)$ by (\ref{Tre-eTr}). Furthermore, by Lemma \ref{lem:csemiration}, (\ref{Tre-eTr}) and Corollary \ref{cor:basis BK algebra}, there is a homomorphism
$$\rho_q \colon \tilde{\mathcal{L}}_q(\Lambda) \to \mcH_q(\Lambda)e(\mcC)$$ sending the generators $X_r, T_r$ to the same named elements, and $f(X)^{-1}\epsilon(\bi)$ with $f(q^\bi)\neq 0$ to $f(X)^{-1}e(\bi)$. Similar to the proof of Theorem \ref{thm:derho}, we arrive at our second main result in this paper for non-degenerate affine Hecke algebras.
\begin{theorem} \label{thm:rho}There is an algebra isomorphism $\tilde{\mathcal{L}}_q(\Lambda)\cong \mcH_q(\Lambda)e(\mcC).$
\end{theorem}

\section{Generalization}
In this section, we give the general and unified definition for the KLR type algebras in the previous Sections.
\subsection{The uniform quivers}
Let $I$ be an abelian group. We consider the loop-free quiver $\Gamma_I$ with vertex set $I$. Denote by $d_{ij}$ the number of arrows $i\to j$. Then $\Gamma_I$ is said to be a {\itshape uniform quiver} if $d_{ij}=d_{i'j'}$ whenever $i-j=i'-j'$. For a fixed $I$, a uniform quiver $\Gamma_I$ corresponds to a map from $I\setminus \{0\}$ to $\mathbb{N}$.

For the cyclic group $I=\mathbb{Z}/e\mathbb{Z}=\{0,\cdots, e-1\}$, where  $e=0\ \mbox{or} \ 2\leq e\in\mathbb{Z}$, the quivers of type $A_\infty$ if $e=0$ or $A_{e-1}^{(1)}$ if $e\geq 2$ are uniform quivers
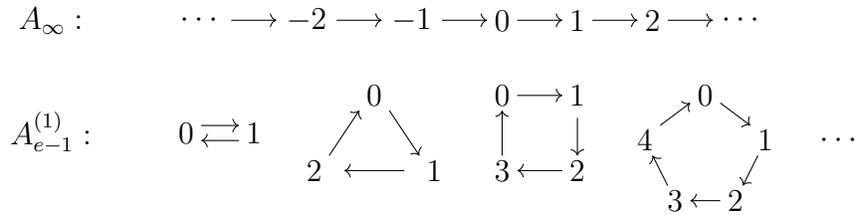
\[
\begin{tikzpicture}
\draw[->] (-4.6,0)--(-4,0);\draw[->](-3.2,0)--(-2.6,0);\draw[->](-1.8,0)--(-1.2,0);\draw[->](-.8,0)--(-.2,0);\draw[->](.2,0)--(.8,0);
\draw[->](1.2,0)--(1.8,0);
\node at (-7,0) {$A_\infty:$};\node at (-5,0) {$\cdots$};\node at (-3.6,0) {$-2$};\node at (-2.2,0){$-1$};\node at (-1,0) {$0$};
\node at (0,0) {$1$};\node at (1,0){$2$};\node at (2.2,0){$\cdots$};

\draw[->](-5,-1.4)--(-4.5,-1.4);\draw[->](-4.5,-1.6)--(-5,-1.6);\draw[->](-3.3,-1.8)--(-2.9,-1.2);\draw[->](-2.5,-1.2)--(-2.1,-1.8);\draw[->](-2.3,-2)--
(-3.1,-2);\draw[->](-1,-1.8)--(-1,-1.2);\draw[->](-.8,-1)--(-.2,-1);\draw[->](0,-1.3)--(0,-1.8);\draw[->](-.2,-2)--(-.8,-2);\draw[->](1.2,-2.2)--(1,-1.8);
\draw[->](1.1,-1.4)--(1.5,-1.1);\draw[->](1.9,-1.1)--(2.3,-1.4);\draw[->](2.4,-1.8)--(2.2,-2.2);\draw[->](1.9,-2.4)--(1.5,-2.4);
\node at (-7,-1.5){$A_{e-1}^{(1)}:$};\node at(-5.2,-1.5){$0$};\node at(-4.3,-1.5){$1$};\node at(-3.5,-2){$2$};\node at(-2.7,-1){$0$};\node at(-1.9,-2){$1$};\node at(-1,-2){$3$};\node at(-1,-1){$0$};\node at (0,-1){$1$};\node at(0,-2){$2$};\node at(1.3,-2.4){$3$};\node at(.9,-1.6){$4$};
\node at(1.7,-1){$0$};\node at(2.5,-1.6){$1$};\node at(2.1,-2.4){$2$};\node at (3.5,-1.6){$\cdots$};
\end{tikzpicture}
\]
and so are the followings:
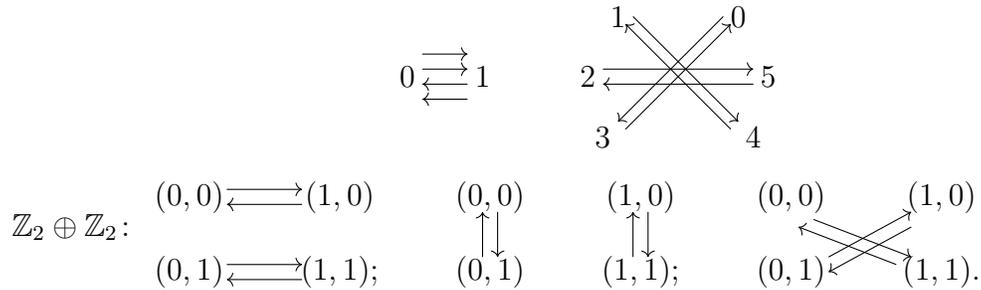
\[
\begin{tikzpicture}
\draw[->] (-.6, .8)--(.8,-0.6);
	\node at (1,-0.8){$4$};
	\draw[->] (.7, -.7)--(-.7,0.7);
	\node at (-.8,0.8){$1$};

	\draw[->] (.6, .8)--(-.8,-0.6);
	\node at (-1,-0.8){$3$};
	\draw[->] (-.7, -.7)--(.7,0.7);
	\node at (.8,0.8){$0$};
	\draw[->] (-1,0.1)--(1,0.1);
	\node at (-1.2,0){$2$};
	\draw[->] (1,-0.1)--(-1,-0.1);
	\node at (1.2,0){$5$};
	
	\draw[->](-3.4,0.1)--(-2.8,0.1);
	\draw[->](-3.4,0.3)--(-2.8,0.3);
	\draw[->](-2.8,-0.1)--(-3.4,-.1);\draw[->](-2.8,-0.3)--(-3.4,-0.3);	
	\node at (-3.6,0){$0$};\node at (-2.6,0){$1$};

	\node at (-8,-2){$\mbZ_2\oplus \mbZ_2\colon$};
	\draw[->](-6,-1.5)--(-5, -1.5);
	\draw[->](-5,-1.7)--(-6,-1.7);
	\node at (-6.5,-1.6){$(0,0)$};\node at (-4.5,-1.6){$(1,0)$};
	\draw[->](-6,-2.5)--(-5, -2.5);
	\draw[->](-5,-2.7)--(-6,-2.7);
	\node at (-6.5,-2.6){$(0,1)$}; \node at (-4.5,-2.6){$(1,1);$};
	\draw[->](-2.4,-1.8)--(-2.4, -2.4);
	\draw[->](-2.6,-2.4)--(-2.6,-1.8);
	\node at (-2.5,-1.6){$(0,0)$};\node at (-.5,-1.6){$(1,0)$};
	\draw[->](-.4,-1.8)--(-.4, -2.4);
	\draw[->](-.6,-2.4)--(-.6,-1.8);
	\node at (-2.5,-2.6){$(0,1)$}; \node at (-.5,-2.6){$(1,1);$};
 \draw[->](1.8,-1.9)--(3.1, -2.4);
	\draw[->](2.9,-2.5)--(1.6,-2);
	\node at (1.5,-1.6){$(0,0)$};\node at (3.5,-1.6){$(1,0)$};
	\draw[->](2,-2.4)--(3.1, -1.8);
	\draw[->](3.1,-2)--(2,-2.6);
	\node at (1.5,-2.6){$(0,1)$}; \node at (3.5,-2.6){$(1,1).$};



	\end{tikzpicture}
	\]
\subsection{The root system and the group map} Denote by $[n]=\{1,2,\cdots,n\}$. Let $\Phi$ be a root system with a simple root system $\Delta=\{\alpha_r \ | \ r\in[n]\}$. Assume that $\mcW$ is the Weyl group of $\Phi$.  Assume $\mcW$ is generated by $\{\sigma_r | r\in [n]\}$ where $\sigma_r$ is the corresponding simple reflection $\sigma_r$ of $\alpha_r$. Let $R$ be a commutative associative unital $\Bbbk$-algebra such that $\mcW$ has an action on it. We fix a {\itshape $\mcW$-map}
$$y\colon \Phi \to R,\quad \alpha\mapsto y_{\alpha}$$ that is, it is a map satisfies $y_{w(\alpha)}=w(y_\alpha)$ for each $w\in \mcW$. We assume $y_{\alpha}$ is not a zero-divisor of $R$ for each $\alpha\in \Phi$. Therefore $R$ is a subalgebra of $R_{\mathrm{Im}y}$ which is the localization of $R$ with respect to $\mathrm{Im}y$. The $\mcW$-action on $R$ extends to an action of $\mcW$ on $R_{\mathrm{Im}y}$ via ring homomorphism. We define the divided difference operators $\parti_r$ on $R_{\mathrm{Im}y}$ to be
\begin{equation}\parti_r(a)=\frac{\sigma_r(a)-a}{y_{\alpha_r}}.
\end{equation}

\subsection{The index set} Let $S$ be a nonempty set. Denote by $S^{*}=\{f\colon S\to I \ | \ f \ \mbox{is a map} \ \}$. Suppose there is an action of $\mcW$ on $S^*$. Then it extends to an actions on $S^{**}=\{h\colon S^*\to I \ | \ h \ \mbox{is a map} \ \}$ by defining
$$w(h)(f)=h(w^{-1}(f))$$
for $w\in \mcW, h\in S^{**}$ and all $f\in S^*$.

Form now on, we fix a $\mcW$-orbit $\mcC$ of $S^*$ and an element $\eta_r$ in $S^{**}$ for each $r\in [n]$. For simplicity, we denote by $\eta_r(\bi):=i_r, \eta_{r1}(\bi)=\sigma_s(\eta_r)(\bi):=i_{r1}, \eta_{r2}(\bi)=\sigma_r\sigma_s(\eta_r)(\bi):=i_{r2}$ for each $\bi\in \mcC$.

\subsection{Coproducts of associated algebras} Let $\Bbbk$ be a commutative ring. Recall in \cite[Subsection 1.4]{BMM}, given two $\Bbbk$-algebras $A_1$ and $ A_2$ with $1$, the {\itshape coproduct} $A_1\sqcup_\Bbbk A_2$ of $A_1$ and $A_2$ is defined to be the quotient of the tensor algebra
$$T(A_1\oplus A_2)=A_1\oplus A_2\oplus A_1\otimes A_1 \oplus A_1\otimes A_2\oplus A_2\otimes A_2\oplus\cdots$$
modulo the ideal generated by all elements of the form $$a_1\otimes b_1-a_1b_1, a_2\otimes b_2-a_2b_2, 1_{A_1}-1_{A_2}$$
where $a_1,b_1\in A_1, a_2,b_2\in A_2$.

\subsection{A KLR-type algebra}
Let $\Gamma_I$ be an uniform quiver with vertex set $I$. For each $m\in I$, we fix a $\mcW$-map $L_{m}\colon \mathrm{Im}y \to R_{\mathrm{Im}y}$ which is defined by
\begin{equation} \label{Lm} L_{m}(y_\alpha)=\begin{cases}y_\alpha^{-1}+\sigma_\alpha(y_\alpha)^{-1} & \text{if $m=0$};\\
y_\alpha^{d_{ij}}\sigma_\alpha(y_\alpha)^{d_{ji}} & \text{if $\exists i\neq j\in I$ such that $m=i-j$}
\end{cases}\end{equation}
for $\alpha\in \Phi$. Denote by $L=(L_m)_{m\in I}$. For simplicity, we denote by $y_r=y_{\alpha_r}, y_{r1}=y_{\sigma_{\alpha_s}(\alpha_r)}, y_{r2}=y_{\sigma_{\alpha_r}\sigma_{\alpha_s}(\alpha_r)}$ for $r\in [n]$. Notice that $L_m$ is decided only by the images of the simple roots $\alpha_r$ by the definition of $\mcW$-map.

\begin{definition} \label{relativeKLR}  Let $A$ be a unital $\Bbbk$-algebra generated by $\{\epsilon(\bi), \psi_r | \bi \in \mcC, r\in [n]\}$. The algebra $\mcR(\Gamma_I,\mcC,R,y)$ is defined to be the coproduct $A \sqcup_{\Bbbk}R$
	together with the following relations for all admissible indices:
	
	(1)  $\sum_{\bi\in \mcC}\epsilon(\bi)=1, \ \epsilon(\bi)\epsilon(\bj)=\delta_{\bi\bj}\epsilon(\bi);$
	
	(2) $a\epsilon(\bi)=\epsilon(\bi)a, \forall a\in R;$
	
	(3) $\psi_r\epsilon(\bi)=\epsilon(\sigma_r(\bi))\psi_r;$	
	
	(4) $\psi_ra\epsilon(\bi)=\sigma_r(a)\psi_r\epsilon(\bi)+\delta_{i_r}^0
	\parti_r(a)\epsilon(\bi), \forall a\in R$;

	(5)  $L_{0}(y_r)\in R$ and
	\begin{align*}
	\label{depsi2}&\psi_r^2\epsilon(\bi)=\begin{cases} -\psi_rL_0(y_r)\epsilon(\bi)& \text{if $i_r=0$},\\
	L_{i_r}(y_r)\epsilon(\bi)& \text{if $i_r\neq 0$};
	\end{cases}
	\end{align*}
	
	(6) if $r\nslash s$, then $\psi_r\psi_s=\psi_s\psi_r$;
	
	(7) if $r-\-s$, then
	\begin{equation*} (\psi_r\psi_s\psi_r-\psi_s\psi_r\psi_s)\epsilon(\bi)=\delta_{i_{s1}}^0\tfrac{L_{i_s}(y_s)-L_{i_r}(y_r)}{y_{s1}}\epsilon(\bi);
	\end{equation*}
	
	(8) if $\xymatrix@C10pt{r\ar@{=}[r]|{\rangle}\ar@{=}[r]&s}$, then
	\begin{align*}&[(\psi_r\psi_s)^2-(\psi_s\psi_r)^2]\epsilon(\bi)\\ &=[\delta_{i_{r1}}^0\tfrac{L_{i_s}(y_s)-L_{i_{s1}}(y_{s1})}{y_{r1}}\psi_r-\delta_{i_{s1}}^0\tfrac{L_{i_r}(y_r)-L_{i_{r1}}(y_{r1})}{y_{s1}}\psi_s\\
&+\delta_{i_r}^0\delta_{i_{r1}}^0\tfrac{L_{i_s}(y_s)-L_{i_{s1}}(y_{s1})}{y_ry_{r1}}]\epsilon(\bi);
	\end{align*}
	
	(9) if $\xymatrix@C10pt{r\ar@3{-}[r]|{\rangle}\ar@3{-}[r]&s}$, then
	\begin{align*}&[(\psi_r\psi_s)^3-(\psi_s\psi_r)^3]\epsilon(\bi)\\
	&=[\delta_{i_{r1}}^0\tfrac{L_{i_s}(y_s)-L_{i_{s2}}(y_{s2})}
	{y_{r1}}\psi_r\psi_s\psi_r-\delta_{i_{s1}}^0\tfrac{L_{i_r}(y_r)-L_{i_{r2}}(y_{r2})}{y_{s1}}\psi_s\psi_r\psi_s \\
	&+\delta_{i_r}^0\delta_{i_{r1}}^0\tfrac{y_rL_{i_{s2}}(y_{s2})-(y_r-y_{r2})L_{i_{s1}}(y_{s1})-y_{r2}L_{i_s}(y_s) }{y_ry_{r1}y_{r2}}\psi_r\psi_s\\
	&+\delta_{i_{s2}}^0\tfrac{L_{i_{r1}}(y_{r1})L_{i_s}(y_s)-L_{i_{r2}}(y_{r2})L_{i_{s1}}(y_{s1})}{y_{s2}}\psi_r-\delta_{i_{r2}}^0\tfrac{L_{i_r}(y_r)L_{i_{s1}}(y_{s1})-L_{i_{r1}}(y_{r1})L_{i_{s2}}(y_{s2})}{y_{r2}}\psi_s\\
	&+\delta_{i_r}^0\delta_{i_{s2}}^0\tfrac{L_{i_{r1}}(y_{r1})L_{i_s}(y_s)-L_{i_{r2}}(y_{r2})L_{i_{s1}}(y_{s1})}{y_ry_{s2}}-\delta_{i_s}^0\delta_{i_{r2}}^0\tfrac{L_{i_r}(y_r)L_{i_{s1}}(y_{s1})-L_{i_{r1}}(y_{r1})L_{i_{s2}}(y_{s2})}{y_sy_{r2}}\\
	&+\delta_{i_{r1}}^0\delta_{i_{s1}}^0\tfrac{L_{i_{r2}}(y_{r2})L_{i_{s}}(y_{s})-L_{i_r}(y_r)L_{i_{s2}}(y_{s2})}{y_{r1}y_{s1}}]\epsilon(\bi).
	\end{align*}
	(10) all the coefficients above are in $R$.
	
\end{definition}
\begin{remark} 	(a) If $R=R_{\mathrm{Im}y}$, then the condition (10) holds automatically. For example, the Lusztig extension in KLR form.

	In the sequel, we take $I=\mathbb{Z}/e\mathbb{Z}$ and $\Gamma_I=A_\infty$ or $A_{e-1}^{(1)}$.

	(b) Take $R=\Bbbk[y_1,\cdots,y_n]$ to be the polynomial ring with indeterminates $\{y_r \ | \ r\in [n]\}$ and $\Bbbk(y_1,\cdots, y_n)$ the corresponding rational functions field. Then there is an action of $\mcW$ on $R$ (by ring automorphism) such that for every $r\in [n]$
	\begin{equation*}
	\sigma_r(y_s)=y_s-a_{sr}y_r.
	\end{equation*}
	Let $y\colon \Phi \to R$ be the map by sending $\sum_{r\in [n]}a_r\alpha_r$ to $\sum_{r\in [n]}a_ry_r$. Then $y$ is a $\mcW$-map and the map $L_m\colon {\rm Im}y\to \Bbbk(y_1,\cdots,y_n)$ in (\ref{Lm}) is given by
	$$L_m(y_r)=\begin{cases}0 & \text{if $m=0$};\\
	-y_r & \text{if $m=1, e\neq 2$};\\
	y_r & \text{if $m=-1, e\neq 2$};\\
	-y_r^2 & \text{if $m=1, e=2$};\\
	1 & \text{if $m\neq 0,\pm 1$}.
	\end{cases}$$
	Take $S=\{1,2,\cdots,n\}$, then $S^*=I^n$ and there is a $\mcW$-action on $I^n$ as defined by (\ref{group action}). Let $\eta_r$ be the map $\eta_r(\bi)=i_r$ for all $\bi\in I^n$ and $r\in [n]$. Then the algebra $\mcR(\Gamma_I,\mcC,R,y)$ is isomorphic to the algebra $\mcR$ as defined in Subsection 3.8.

	(c) Let $R=\{\tfrac{f}{g}\ | \ f,g\in \Bbbk[y_1,\cdots,y_n] \ \mbox{with} \ g(0)\neq 0\}\subset \Bbbk(y_1,\cdots,y_n)$. Then there is an action of $\mcW$ on $R$  (via the ring automorphism) induced by
	$$\sigma_r(y_s)=1-(1-y_s)(1-y_r)^{-a_{sr}}.$$
	Let $y\colon \Phi\to R$ be the map decided by sending $\sum_{r\in [n]}a_r\alpha_r$ to $ \prod_{r\in [n]}(1-y_r)^{a_r}$. then $y$ is a $\mcW$-map and the map $L_m\colon {\rm Im}y\to K(y_1,\cdots,y_n)$ in (\ref{Lm}) is given by
	$$L_m(y_r)=\begin{cases}0 & \text{if $m=0$};\\
	\frac{y_r}{y_r-1} & \text{if $m=1, e\neq 2$};\\
	y_r & \text{if $m=-1, e\neq 2$};\\
	\frac{y_r^2}{y_r-1} & \text{if $m=1, e=2$};\\
	1 & \text{if $m\neq 0,\pm 1$}.
	\end{cases}$$
	Take $S=\{1,2,\cdots,n\}$, then $S^*=I^n$ and there is a $\mcW$-action on $I^n$ as defined by (\ref{group action}).	Let $\eta_r$ be the map $\eta_r(\bi)=i_r$ for all $\bi\in I^n$ and $r\in[n]$. In this case, the algebra $\mcR(\Gamma_I,\mcC,R,y)$ is isomorphic to the BK subalgebra $\tilde{\mathcal{L}}_q$ as defined in Subsection 4.7.

	(d) Take $\mcW=S_n$ to be the symmetric group. Let $R=\Bbbk[y_1,\cdots,y_{n+1}]$ be the polynomial ring, then there is an action of $S_n$ on $R$ by permuting variables. Let $y\colon \Phi \to R$ be the map by sending $\sum_{r\in [n]}a_r\alpha_r$ to $\sum_{r\in [n]}a_r(y_r-y_{r+1})$. Then $y$ is a $\mcW$-map and the map $L_m\colon {\rm Im}y\to \Bbbk(y_1,\cdots,y_{n+1})$ in (\ref{Lm}) is given by
	$$L_m(y_r)=\begin{cases}0 & \text{if $m=0$};\\
	y_{r+1}-y_{r} & \text{if $m=1, e\neq 2$};\\
	y_{r}-y_{r+1} & \text{if $m=-1, e\neq 2$};\\
	(y_{r+1}-y_{r})(y_{r}-y_{r+1}) & \text{if $m=1, e=2$};\\
	1 & \text{if $m\neq 0,\pm 1$}.
	\end{cases}$$
	Take $S=\{1,2,\cdots,n+1\}$, then $S^*=I^{n+1}$ and there is an $S_n$-action on $S^*$ by the place permutation $:w(\bi)_s=i_{w^{-1}(s)}$. Let $\eta_r$ be the map $\eta_r(\bi)=i_r-i_{r+1}$ for all $\bi\in I^{n+1}$ and $r\in [n]$. Then the algebra $\mcR(\Gamma_I,\mcC,R,y)$ is isomorphic to KLR algebra of type A as defined in \cite[Subsection 4.4]{KLi}.

   (e) For symmetry, in Definition 5.3, the term $-\delta_{i_s}^0\delta_{i_{s1}}^0\tfrac{L_{i_r}(y_r)-L_{i_{r1}}(y_{r1})}{y_sy_{s1}}$  should be added to the sum in (8), and  $-\delta_{i_s}^0\delta_{i_{s1}}^0\tfrac{y_sL_{i_{r2}}(y_{r2})-(y_s-y_{s2})L_{i_{r1}}(y_{r1})-y_{s2}L_{i_r}(y_r) }{y_sy_{s1}y_{s2}}\psi_s\psi_r$  should be added  to the sum  in (9). However, in all the examples above they are always zero.

\end{remark}

\end{document}